\documentclass[11pt]{article}

\usepackage[utf8]{inputenc}
\usepackage[T1]{fontenc}
\usepackage{lmodern}
\usepackage{amsmath, amssymb, amsthm}
\usepackage{mathrsfs}								%Script-Buchstaben
\usepackage{enumitem}								%Elegante Nummerierung
\usepackage{xcolor}
\usepackage[linktocpage]{hyperref}
\usepackage{xifthen}								%Für \isempty
\usepackage{url}
\usepackage{booktabs}								%Tabellen mit top-, mid- und bottomrule
\usepackage{caption}
\usepackage{bbm}									%Indikatorfunktion
\usepackage{stmaryrd}								%Doppelslash
\usepackage{todonotes}

\numberwithin{equation}{section}
\usepackage[margin=3.5cm]{geometry}

\newcommand{\nlb}{{\ensuremath{\textnormal{b}}}}
\newcommand{\nlc}{{\ensuremath{\textnormal{c}}}}
\newcommand{\nld}{{\ensuremath{\textnormal{d}}}}

\newcommand{\rmc}{{\ensuremath{\mathrm{c}}}}
\newcommand{\rmd}{{\ensuremath{\mathrm{d}}}}
\newcommand{\rme}{{\ensuremath{\mathrm{e}}}}

\newcommand{\rmR}{{\ensuremath{\mathrm{R}}}}

\usepackage{ifthen}

\newcommand{\sfp}{{\ensuremath{\mathsf{p}}}}
\newcommand{\sfq}{{\ensuremath{\mathsf{q}}}}

\newcommand{\sfv}{{\ensuremath{\mathsf{v}}}}

\newcommand{\sfC}{{\ensuremath{\mathsf{C}}}}

\newcommand{\sfE}{{\ensuremath{\mathsf{E}}}}

\newcommand{\sfP}{{\ensuremath{\mathsf{P}}}}

\newcommand{\scrE}{{\ensuremath{\mathscr{E}}}}
\newcommand{\scrF}{{\ensuremath{\mathscr{F}}}}

\newcommand{\scrH}{{\ensuremath{\mathscr{H}}}}

\newcommand{\scrR}{{\ensuremath{\mathscr{R}}}}

\newcommand{\bbE}{{\ensuremath{\mathbb{E}}}}

\newcommand{\bbN}{{\ensuremath{\mathbb{N}}}}

\newcommand{\bbP}{{\ensuremath{\mathbb{P}}}}

\newcommand{\bbR}{{\ensuremath{\mathbb{R}}}}

\newcommand{\N}{\bbN}						%Natürliche Zahlen
\newcommand{\R}{\bbR}						%Reelle Zahlen
\renewcommand{\d}{\,\mathrm{d}}				%Kleines Integral-d

\let\limsup\undefined

\DeclareMathOperator*{\limsup}{lim\,sup}		%Limes superior
\DeclareMathOperator{\sgn}{sgn}				%Signum

\let\originalleft\left			
\let\originalright\right
\renewcommand{\left}{\mathopen{}\mathclose\bgroup\originalleft}
\renewcommand{\right}{\aftergroup\egroup\originalright}

\newcommand{\onto}[2]{\left. #1\right\vert_{#2}}					%Einschränkung
\newcommand{\mapdef}[3][]{\ifthenelse{\isempty{#1}}{#2\quad\longmapsto\quad #3}{#1\colon\quad #2\quad\longmapsto\quad #3}}		%Abbildungsdefinition, abgesetzt
\newcommand{\der}[2][]{\ifthenelse{\isempty{#1}}{\frac{\nld}{\nld #2}}{\left.\frac{\nld}{\nld #2}\right\vert_{#1}}}				%Ableitungsoperator
\makeatletter
\newcommand{\forevery}[1]{\quad\text{for every~}#1\checknarg}		%Für alle
	\newcommand{\checknarg}{\@ifnextchar\bgroup{\gobblenarg}{}}
	\newcommand{\gobblenarg}[1]{\@ifnextchar\bgroup{,\ \! #1\gobblenarg}{,\ \! #1}}

\theoremstyle{definition}
\newtheorem{bump}{Bump}[section]
\newtheorem{remarkx}[bump]{Remark}
\newtheorem{examplex}[bump]{Example}
\newenvironment{remark}{\pushQED{\qed}\remarkx}{\popQED\endremarkx}		%Bemerkung mit schwarzem Kästchen		
\newenvironment{example}{\pushQED{\qed}\examplex}{\popQED\endexamplex}		%Beispiel mit schwarzem Kästchen

\theoremstyle{plain}
\newtheorem{theorem}[bump]{Theorem}

\newtheorem{definition}[bump]{Definition}
\newtheorem{lemma}[bump]{Lemma}

\newtheoremstyle{cited}
	{\topsep}		%Space above
	{\topsep}		%Space below
	{\itshape}		%Body font
  	{}				%Indent amount
	{\bfseries}		%Theorem head font
	{\textbf{.}}	%Punctuation after Theorem head
	{.5em}			%Space after Theorem head
	{\thmname{#1} \thmnumber{#2} \thmnote{\normalfont#3}}		%Theorem head

\theoremstyle{cited}			%Umgebungen für Zitate (d.h. ohne runde Klammern)
\newtheorem{citetheorem}[bump]{Theorem}

\newtheorem{citelemma}[bump]{Lemma}

\makeatletter
\let\@fnsymbol\@arabic	 		%Arabische Ziffern auf der Titelseite
\makeatother

\makeatletter
\def\nonumberfootnote{\xdef\@thefnmark{}\@footnotetext}			%Fußnote ohne Nummer
\makeatother

\renewenvironment{abstract}{\begin{quotation}\small
\noindent\textbf{Abstract}\quad}
{\end{quotation}}							%Neue Abstract-Umgebung

\newcommand{\mms}{\mathit{M}}				%Basisraum
\newcommand{\met}{\rho}						%Metrik
\newcommand{\Rmet}{g}						%Riemannian metric
\newcommand{\meas}{\mathfrak{m}}			%Referenzmaß
\newcommand{\Id}{\mathrm{Id}}				%Identitätsabbildung
\newcommand{\RCD}{\mathrm{RCD}}				%RCD-Bedingung
\newcommand{\CD}{\mathrm{CD}}				%CD-Bedingung
\newcommand{\bounded}{\nlb}					%Beschränkt
\newcommand{\comp}{\nlc}					%Kompakt
\newcommand{\loc}{\mathrm{loc}}				%Lokal
\newcommand{\Ric}{\mathrm{Ric}}				%Ricci-Krümmungstensor
\newcommand{\Cont}{C}					%Stetige Funktionen
\newcommand{\Ell}{\mathit{L}}				%Lebesgue-integrierbare Funktionen
\newcommand{\Lip}{\mathrm{Lip}}				%Lipschitz-stetige Funktionen/Kurven
\newcommand{\Sob}{\mathit{W}}				%Sobolev-Raum
\newcommand{\Geo}{\mathrm{Geo}}				%Geodätische
\DeclareMathOperator{\Hess}{Hess}			%Hessesche
\newcommand{\ChHeat}{\sfP}					%Wärmefluss der Cheeger-Energie
\newcommand{\WHeat}{\vec{\sfP}}					%Wärmefluss der Entropie
\newcommand{\Schr}[1]{\sfP^{#1}}			%Schrödinger-Halbgruppe
\newcommand{\eval}{\mathsf{ev}}					%Evaluationsabbildung
\newcommand{\push}{\sharp}					%Bildmaßoperator

\newcommand{\kk}{\underline{k}}				%Altes k quer
\newcommand{\kkk}{\underline{\ell}}					%Altes k quer
\newcommand{\B}{X}						%Brownsche Bewegung
\newcommand{\Q}{Q}
\newcommand{\W}{W}
\newcommand{\NN}{N}
\newcommand{\M}{M}
\newcommand{\One}{\mathbbm{1}}				%Indikatorfunktion
\renewcommand{\lll}{{\underline{\ell}}}		%l quer
\newcommand{\Par}{\sslash}					%Paralleltransport
\newcommand{\Cut}{\mathrm{cut}}				%Cutlocus
\newcommand{\sdist}{\rho^{\pm}}		%Signierte Distanzfunktion
\newcommand{\Ball}{\mathsf{B}}				%Ball

\newcommand{\pa}{\sslash}

\newcommand{\Kato}{\mathcal{K}}

\begin{document}

\title{Heat flow regularity, Bismut--Elworthy--Li's derivative formula, and pathwise couplings on Riemannian manifolds\\with Kato bounded Ricci curvature}
\author{Mathias Braun\thanks{University of Bonn, Institute for Applied Mathematics, Endenicher Allee 60, 53115 Bonn, Germany, \textsf{braun@iam.uni-bonn.de}. Funded by the European Research Council through the ERC-AdG ``RicciBounds'', ERC project 10760021.}\and Batu Güneysu\thanks{University of Bonn, Mathematical Institute, Endenicher Allee 60, 53115 Bonn, Germany, \textsf{gueneysu@math.uni-bonn.de}.}}
\date{\today}
\nonumberfootnote{\textit{Key words and phrases.} Ricci curvature, Bismut--Elworthy--Li formula, coupling, Kato class.}
\maketitle

\begin{abstract} 
We prove that if the Ricci tensor $\Ric$ of a geodesically complete Riemannian manifold $\mms$, endowed with the Riemannian distance $\met$ and the Riemannian measure $\meas$, is bounded from below by a continuous function $k\colon\mms\to\R$ whose negative part $k^-$ satisfies, for every $t>0$, the exponential integrability condition
\begin{equation*}
\sup_{x\in\mms} \bbE\Big[\rme^{\int_0^t k^-(\B_r^x)/2\d r}\,\One_{\{t < \zeta^x\}}\Big] < \infty,
\end{equation*}
then the lifetime $\zeta^x$ of Brownian motion $X^x$ on $\mms$ starting in any $x\in \mms$ is a.s.~infinite. This assumption on $k$ holds if $k^-$ belongs to the Kato class of $\mms$. We also derive a  Bismut--Elworthy--Li derivative formula for $\nabla \ChHeat_tf$ for every $f\in L^\infty(\mms)$  and $t>0$ along the heat flow $(\ChHeat_t)_{t\geq 0}$ with generator $\Delta/2$, yielding its $\Ell^\infty$-$\Lip$-regularization as a corollary.

Moreover, given the stochastic completeness of $\mms$, but without any assumption on $k$ except continuity, we prove the equivalence of lower boundedness of $\Ric$ by $k$ to the existence, given any $x,y\in\mms$, of a coupling $(\B^x,\B^y)$ of Brownian motions on $\mms$ starting in $(x,y)$ such that a.s.,
\begin{equation*}
\met\big(\B_t^x,\B_t^y\big) \leq \rme^{-\int_s^t \kk\left(\B_r^x,\B_r^y\right)/2\d r}\,\met\big(\B_s^x,\B_s^y\big)
\end{equation*}
holds for every $s,t\geq 0$ with $s\leq t$, involving the ``average'' $\smash{\kk(u,v) := \inf_\gamma \int_0^1 k(\gamma_r)\d r}$ of $k$ along geodesics from $u$ to $v$. 

Our results generalize to weighted Riemannian manifolds, where the Ricci curvature is replaced by the corresponding Bakry--Émery Ricci tensor.
\end{abstract}
\clearpage
\tableofcontents

\section{Main results}

Let $(\mms,\Rmet)$ be a smooth, geodesically complete, noncompact, connected Riemannian manifold without boundary. The metric  $\langle\cdot,\cdot\rangle := \Rmet(\cdot,\cdot)$  induces the Riemannian distance $\met$ and the Riemannian measure $\meas$. W.r.t.~$\rho$, we write $\Ball_r(x)$ for the open ball of radius $r>0$ around $x\in\mms$, $\Lip(\mms)$ for the space of real-valued Lipschitz functions on $\mms$, and $\Lip(f)$ for the Lipschitz constant of any $f\in\Lip(\mms)$. All appearing vector spaces of functions and sections of bundles are considered as being real and, unless explicitly stated otherwise, all appearing Lebesgue and Sobolev spaces are understood w.r.t.~$\meas$. With the usual abuse of notation, the fiberwise norm both on $T\mms$ and $T^*\mms$ is $\vert\cdot\vert := \langle\cdot,\cdot\rangle^{1/2}$. Let $\nabla$ be the Levi-Civita connection on $\mms$ and $\Ric$ be the induced Ricci curvature. We recall that by geodesic completeness, the Laplace--Beltrami operator $\Delta$ is an essentially self-adjoint operator in $L^2(\mms)$ when defined initially on smooth compactly supported functions \cite{strichartz1983}, and thus admits a unique -- non-relabeled -- self-adjoint extension. Let $(\ChHeat_t)_{t\geq 0}$ be the heat flow in $L^2(\mms)$ with generator $\Delta/2$, i.e.~$\ChHeat_t := \rme^{t\Delta/2}$ via spectral calculus. For every $x\in \mms$, let $X^x\colon [0,\zeta^x)\times \Omega\to \mms$ be a corresponding adapted diffusion process (\emph{Brownian motion}) starting at $x\in\mms$ with lifetime $\zeta^x$, defined on a filtered probability space $(\Omega,\scrF_*,\bbP)$, see \cite{elworthy1982, hsu2002, ikeda1981, wang2014} for particular constructions of $X^x$.

Throughout, we fix a continuous function $k\colon \mms \to\R$. We write ``$\Ric\geq k$ on $\mms$'' if
\begin{equation*}
\Ric(x)(\xi,\xi)\geq k(x)\,\vert\xi\vert^2\quad\text{for every }x\in \mms,\ \xi\in T_x\mms.
\end{equation*}
The goal of this paper is to study the previous condition, where the negative part $k^-$ of $k$, with $k^-(x) := -\min\{k(x),0\}$, obeys the integrability assumption
\begin{equation}\label{Eq:Exp integr assumption}
\sfC_t  <\infty\quad\text{for every }t>0,\quad\text{where}\quad \sfC_t := \sup_{x\in\mms} \bbE\Big[\rme^{\int_0^t k^-(\B_r^x)/2 \d r}\,\One_{\{t < \zeta^x\}}\Big].
\end{equation}

Our main results come in two groups. First, we study analytic and probabilistic \emph{consequences} of the assumption $\Ric\geq k$ on $\mms$ if $k$ satisfies (\ref{Eq:Exp integr assumption}), as described in Section \ref{Sub:Consequ} and stated in Theorem \ref{Th:First theorem} and Theorem \ref{Cor:Dynkin}. Along with this, we treat an explicit class of $k$ for which \eqref{Eq:Exp integr assumption} holds, the so-called \emph{Kato decomposable} ones, and highlight a general condition for $k$ to obey the latter property, Theorem \ref{ayx}. Second, we give equivalent \emph{characterizations} of the condition $\Ric\geq k$ on $\mms$, which are summarized in Section \ref{Sub:Char}, see Theorem \ref{Th:Equivalences} therein, and mostly do not even require \eqref{Eq:Exp integr assumption}.

Besides \cite{erbar2020,guneysu2019}, our article is among the first to systematically study analytic and probabilistic consequences of variable lower Ricci bounds -- and equivalent characterizations of these -- which are not uniformly bounded from below and do not underlie geometric growth conditions. We also stress our novel general sufficient condition from Theorem \ref{ayx} to determine whether a given variable Ricci curvature lower bound is Kato decomposable, while the -- albeit more general -- condition \eqref{Eq:Exp integr assumption} is in general hard to verify directly. Lastly, our equivalence result improves upon previously known ones especially because it involves a pathwise coupling estimate which has just recently been introduced in a slightly different framework \cite{braun2019}.

\subsection{Consequences of variable lower Ricci bounds}\label{Sub:Consequ}

To formulate our first result, given an initial point $x\in\mms$, let $\pa^x$ denote the stochastic parallel transport w.r.t.~$\nabla$ along the sample paths of $\B^x$, i.e.~$\smash{\pa^x_t\colon T_x\mms \to T_{\B^x_t}\mms}$ for all $t\in [0,\zeta^x)$, let the process $\Q^x\colon [0,\zeta^x)\times\Omega\to \mathrm{End}(T_x\mms)$ be defined as the unique solution to the pathwise ordinary differential equation
\begin{equation}\label{Eq:Q ODE}
\rmd \Q_s^x = -\frac{1}{2}\,\Q_s^x\,(\pa_s^x)^{-1}\,\Ric(\B_s^x)\pa_s^x\!\d s,\quad \Q_0^x=\Id_{T_x\mms},
\end{equation}
where $\Ric(X_s^x)$ is regarded as an element of $\smash{\mathrm{End}(T_{X_s^x}\mms)}$. Let $\smash{\W^x\colon [0,\zeta^x)\times\Omega\to T_x\mms}$ denote the anti-development of $\B^x$, a canonically given Euclidean Brownian motion on $T_x\mms$. See \cite{elworthy1982, hsu2002, ikeda1981, wang2014} for details.

\begin{theorem}\label{Th:First theorem} Let $k\colon \mms\to\R$ be a continuous function satisfying \eqref{Eq:Exp integr assumption} and assume that $\Ric\geq k$ on $\mms$. Then
\begin{enumerate}[leftmargin=1.25cm,label=\textnormal{(\roman*)}]
\item $\mms$ is \emph{stochastically complete}, i.e.~
\begin{equation*}
\bbP\big[\zeta^x = \infty\big] = 1\forevery{x\in\mms},
\end{equation*}
\item for every $f\in \Ell^\infty(\mms)$ and every $t>0$, we have \emph{Bismut--Elworthy--Li's derivative formula}
\begin{equation*}
\big\langle\nabla \ChHeat_t f(x),\xi\big\rangle =\frac{1}{t}\,\bbE\left[ f(\B_t^x)\int^{t}_0 \!\big\langle \Q_s^x \xi,\rmd\W_s^x\big\rangle \right]\quad\text{ for every }x\in\mms,\ \xi\in T_x\mms,
\end{equation*}
where the stochastic integral inside the expectation is understood in It\^{o}'s sense, and
\item for every $t >0$, one has the \emph{$\Ell^\infty$-$\Lip$-regularization property} $\ChHeat_t\colon \Ell^{\infty}(\mms)\to \Lip(\mms)$ with
\begin{equation*}
\Lip(\ChHeat_tf) \leq \sqrt{8}\,t^{-1/2}\,\sup_{x\in\mms}\bbE\Big[ \rme^{\int_0^t k^-(\B_r^x)/2 \d r}\Big]\,\Vert f\Vert_{\Ell^\infty}\forevery{f\in \Ell^\infty(\mms)}.
\end{equation*}
\end{enumerate}
\end{theorem}

Before further commenting on Theorem \ref{Th:First theorem} and its proof, in order to make more refined statements, we introduce the following definition.

\begin{definition}\label{Def:Dynkin and Kato} \begin{enumerate}[leftmargin=1.25cm,label=\textnormal{(\roman*)}]
		\item The \emph{Kato class} $\Kato(\mms)$ of $\mms$ is the linear space of all Borel functions $\sfv\colon \mms\to\R$ such that
		\begin{equation*}
		\lim_{t\downarrow 0}\,\sup_{x\in\mms}\int_0^t\bbE\big[\big\vert \sfv(\B_r^x)\big\vert\big]\d r = 0.
		\end{equation*}
		\item A Borel function $\sfv\colon \mms \to \R$ is called \emph{Kato decomposable} if it belongs to $\smash{\Ell_\loc^1(\mms)}$ and $\sfv^-$ belongs to $\Kato(\mms)$.
	\end{enumerate}
\end{definition}

Kato (decomposable) functions have been studied in great detail in the literature in the context of (scalar) Schrödinger operators, see \cite{aizenman1982,berthier,chung, guneysu2017, stollmann1996, sturm1994} and the references therein. The survey \cite{rose2020} provides a concise overview over the use of Kato decomposability in the context of Riemannian manifolds and its connections to semigroup domination. A  detailed study of the Kato class and the induced Schrödinger semigroups corresponding to a large class of Hunt processes can be found in \cite{demuth}. (Unfortunately, the authors of \cite{demuth} assume throughout that the underlying process has the Feller property, and it is not known whether this property holds in the situation of Theorem \ref{Th:First theorem}, or Theorem \ref{Cor:Dynkin} below. This is why at many places we are going to rely on the results from \cite{guneysu2017} instead which are formulated for arbitrary Riemannian manifolds.) In connection with lower Ricci bounds, Kato decomposable functions have been introduced in \cite{guneysu2015} in the context of BV functions. They have been considered further recently in \cite{carron2019,rose2019} in the context of heat kernel, Betti number and eigenvalue estimates, and in \cite{ouhabaz} within the study on $\Ell^p$-properties of heat semigroups on forms. See also \cite{guneysu2019}, which treats some probabilistic and geometric aspects of molecular Schrödinger operators under Kato assumptions. 

For the convenience of the reader, we have collected some important properties of Kato decomposable functions in Appendix \ref{Sec:Kato}. In particular, note that in view of
\begin{equation*}
\bbE\big[\big\vert \sfv(\B_r^x)\big\vert\big] \leq \Vert \sfv\Vert_{\Ell^\infty}\forevery{x\in \mms,\ r\geq 0},
\end{equation*}
it follows that $\Ell^\infty(\mms) \subset\Kato(\mms)$. More generally, in view of an explicit Example \ref{Exa}, we provide the following criterion in Section \ref{Check}, for which we denote by $\Xi\colon\mms\to\R$ the function $\Xi(x):= \meas[\Ball_1(x)]^{-1}$, and by $L^p_{\Xi}(\mms)$ the $L^p$-space w.r.t.~$\Xi\,\meas$. 

\begin{theorem}\label{ayx} Assume that $\dim(\mms)\geq 2$, that $\mms$ is quasi-isometric to a complete Riemannian manifold whose Ricci curvature is bounded from below by a constant, and that $k^-\in L^p_{\Xi}(\mms)+L^{\infty}(\mms)$ for some $p\in (\dim(\mms)/2,\infty)$. Then $k$ is Kato decomposable.
\end{theorem}

One key feature for us about functions $\sfv\in \Kato(\mms)$ is that they always satisfy 
$$
\sup_{x\in\mms} \bbE\Big[\rme^{\int_0^t \sfv(\B_r^x)/2 \d r}\,\One_{\{t < \zeta^x\}}\Big] <\infty\quad\text{locally uniformly in $t\in [0,\infty)$. }
$$
This is known as \emph{Khasminskii's lemma}, see Lemma \ref{Le:Khasminskii}. In particular, since $\Kato(\mms)$ is a linear space, we have the following link of Kato decomposability to \eqref{Eq:Exp integr assumption}.

\begin{lemma}\label{Le:Khas} Assume that $k$ is a Kato decomposable function. Then for every $q\in [1,\infty)$, the exponential integrability \eqref{Eq:Exp integr assumption} holds with $k$ replaced by $qk$.
\end{lemma}

This is ultimately the key behind the following result which states that in this case, Bismut--Elworthy--Li's derivative formula holds on an $L^p$-scale. 

\begin{theorem}\label{Cor:Dynkin} Assume $k\colon \mms\to\R$ is a continuous Kato decomposable function satisfying $\Ric\geq k$ on $\mms$. Then \eqref{Eq:Exp integr assumption} is satisfied, and moreover, Bismut--Elworthy--Li's derivative formula from Theorem \ref{Th:First theorem} holds for every $p\in (1,\infty]$ and every $f\in L^p(\mms)$.
\end{theorem}

The proof of (i) in Theorem \ref{Th:First theorem} can be found in Section \ref{Sec:Stochastic completeness}, while (ii) and (iii) as well as Theorem \ref{Cor:Dynkin} are studied in Section \ref{Sec:Bismut}. 

Let us collect some bibliographical comments on Theorem \ref{Th:First theorem} and Theorem \ref{Cor:Dynkin}.

In the framework of uniform bounds from below on the Ricci curvature, (i) in Theorem \ref{Th:First theorem} is due to \cite{yau}. On weighted Riemannian manifolds -- on which the Ricci tensor is always replaced by the corresponding Bakry--Émery Ricci tensor, see Section \ref{Extensions} -- the non-explosion for the induced diffusion processes under uniform lower Ricci bounds has been obtained by \cite{bakry1986}. In connection with \eqref{Eq:Exp integr assumption}, also for weighted Riemannian manifolds, the latter result has been extended by \cite{li1994} using an approach via stochastic and Hessian flows. In fact, the corresponding condition at page 423 of \cite{li1994} is implied by our condition \eqref{Eq:Exp integr assumption}. Once we have established all necessary intermediate results, our proof then closely follows the lines in \cite{bakry1986} (which is also worked towards in \cite{li1994}). For different, more geometric non-explosion criteria in terms of distance functions, see \cite{wang2014} and the references therein. A nonsmooth result similar to (i) -- however assuming a Kato- or rather a Dynkin-type \cite{stollmann1996, sturm1994} lower bound instead of only \eqref{Eq:Exp integr assumption} -- has recently been treated in \cite{erbar2020}.

Formula (ii) in Theorem \ref{Th:First theorem} has first appeared in \cite{bismut1984} in the compact case. In the noncompact case, this result, as well as Theorem \ref{Cor:Dynkin}, have been proven in \cite{EL2,EL} under more general assumptions than \eqref{Eq:Exp integr assumption} using the slightly different technique of stochastic derivative flows. We also refer to \cite{driver2001} for similar treatises for heat semigroups over vector bundles, and also \cite{hsu2002, wang2014} for similar results under more geometric conditions on the lower bound of $\Ric$. Remarkably, localized versions of the Bismut--Elworthy--Li derivative formula hold without any assumptions on the geometry of the manifold, see e.g.~\cite{T1,T3,T2}. 

The $\Ell^\infty$-$\Lip$-regularization (iii) from Theorem \ref{Th:First theorem} is a corollary of (ii), thus indicating the importance of the latter in studying further regularity properties of $(\ChHeat_t)_{t\geq 0}$. In fact, \emph{local} versions of (iii) are already known even without the assumption \eqref{Eq:Exp integr assumption} on $k$ \cite{T3, wang2014}, with slightly different estimates on $\Lip(\ChHeat_t f)$ involving locally uniform lower bounds on $\Ric$. (The proof uses the above mentioned local derivative formula.) Outside the smooth scope, a similar property as (iii) is known on $\RCD(K,\infty)$ spaces \cite{ags}. This setting allows for more flexibility in the variety of spaces (metric measure spaces), but is still restricted to uniform lower Ricci bounds, formulated in a synthetic sense \cite{lott2009, sturm2006}.

\subsection{Characterizations of variable lower Ricci bounds}\label{Sub:Char}

We now come to our second main result, i.e.~several equivalent characterizations of lower Ricci bounds, which we shortly introduce. 

The closest characterization of $\Ric\geq k$ on $\mms$ is the \emph{$\Ell^1$-Bochner inequality} which is related to the Ricci curvature of $\mms$ by the following well-known Bochner formula: given any open $U\subset\mms$, for every $f\in \Cont^\infty(U)$ we have
\begin{equation}\label{Eq:Bochner formula}
\Delta\frac{\vert\nabla f\vert^2}{2} = \langle \nabla \Delta f,\nabla f\rangle + \vert\!\Hess f\vert^2 + \Ric(\nabla f,\nabla f)\quad\text{on }U.
\end{equation}

We also derive a one-to-one connection between lower Ricci bounds by $k$ and the existence of certain couplings of Brownian motions on $\mms$. Here, if $\mms$ is stochastically complete, then given $x,y\in\mms$, by a \emph{coupling of Brownian motions starting in $(x,y)$}, we understand an $\mms\times\mms$-valued stochastic process $(\B^x,\B^y)\colon [0,\infty)\times\Omega\to\mms\times\mms$ on some filtered probability space $(\Omega,\scrF_*,\bbP)$ such that $\B^x$ and $\B^y$ are Brownian motions on $\mms$ starting in $x$ and $y$, respectively. To formulate an appropriate pathwise coupling estimate, we denote by $\Geo(\mms)$ the set of minimizing geodesics $\gamma \colon [0,1]\to\mms$, and define the lower semicontinuous function $\kk\colon \mms\times\mms\to\R$ by
\begin{equation}\label{Eq:kk definition}
\kk(u,v) :=  \inf\!\bigg\lbrace\! \int_0^1 k(\gamma_r) \d r: \gamma\in \Geo(\mms),\ \gamma_0=u,\ \gamma_1=v\bigg\rbrace.
\end{equation}
Observe that $k$ can be recovered from the diagonal values of $\kk$, i.e.~$k(u)= \kk(u,u)$ for every $u\in\mms$. The key feature about $\kk$ is that it provides a way to avoid cut-loci.

\begin{theorem}\label{Th:Equivalences} Let $k\colon \mms\to\R$ be a continuous function satisfying \eqref{Eq:Exp integr assumption}. Then the following conditions are equivalent:
\begin{enumerate}[leftmargin=1.25cm,label=\textnormal{(\roman*)}]
\item we have $\Ric \geq k$ on $\mms$,
\item the \emph{$\Ell^1$-Bochner inequality w.r.t.~$k$} is satisfied, i.e.~for every $f\in C_\comp^\infty(\mms)$,
\begin{align}\label{Eq:1-Bochner}
\Delta\vert\nabla f\vert - \vert\nabla f\vert^{-1}\,\langle \nabla\Delta f,\nabla f\rangle \geq k\,\vert\nabla f\vert\quad \text{on }\{\vert\nabla f\vert\neq 0\},
\end{align}
\item we have the \emph{pathwise coupling property w.r.t.~$k$}, i.e.~$\mms$ is stochastically complete and for every $x,y\in\mms$, there exists a coupling $(\B^x,\B^y)$ of Brownian motions on $\mms$ starting in $(x,y)$ such that a.s., we have
\begin{equation*}
\met\big(\B_t^x,\B_t^y\big) \leq \rme^{-\int_s^t \kk\left(\B_r^x,\B_r^y\right)/2\d r}\met\big(\B_s^x,\B_s^y\big)\forevery{s,t\geq 0}\text{ with }s\leq t.
\end{equation*}
\end{enumerate}
\end{theorem}

%Here and in the sequel, the Markov property for every respective process under consideration is understood w.r.t.~its canonically given filtration. The statement of Theorem \ref{Th:Equivalences} is still true if one does not require the Markov property of $(\B^x,\B^y)$.

\begin{remark}\label{Re:Ext of cpl} Thanks to the local, respectively pathwise, nature of the statements, ``(iii) $\Longrightarrow$ (ii)'' and ``(ii) $\Longleftrightarrow$ (i)'' are even true without \eqref{Eq:Exp integr assumption}. Moreover, under the \emph{a priori} assumption of stochastic completeness, ``(i) $\Longrightarrow$ (iii)'' is satisfied without \eqref{Eq:Exp integr assumption}. For a slightly more general version of ``(iii) $\Longrightarrow$ (ii)'', see Remark \ref{Re:Ru} below.
\end{remark}

We prove ``(ii) $\Longrightarrow$ (i)'' in Section \ref{Sec:3}, ``(i) $\Longrightarrow$ (iii)'' in Section \ref{Sec:(i) => (iii)} and ``(iii) $\Longrightarrow$ (ii)'' in Section \ref{Sec:(iii) => (ii)}. For Kato decomposable functions $k$, another equivalent characterization of $\Ric \geq k$ on $\mms$ in terms of the \emph{$\Ell^1$-gradient estimate} is discussed in Section \ref{Sec:2}. 

Again, some bibliographical comments are in order.

In the abstract framework of \cite{erbar2020}, the equivalence ``(i) $\Longleftrightarrow$ (ii)'' --  with (ii) in a weak formulation -- together with their equivalence to (a nonsmooth version of) the $\Ell^1$-gradient estimate from Theorem \ref{Th:Sc} has been shown independently. 

The pathwise estimate appearing in (iii), as well as the equivalence of (iii) to lower Ricci bounds, extends similar results from \cite{braun2019}, %(the Markov property of the constructed coupling has not been established therein), 
where analogous equivalences have been established in the synthetic framework of (infinitesimally Hilbertian) $\CD(k,\infty)$ spaces with lower semicontinuous, lower bounded variable Ricci bounds $k\colon\mms\to\R$ (see also \cite{sturm2015}). Even for the smooth case, the stated pathwise inequality involving the function $\kk$ has been firstly introduced in \cite{braun2019}. (Although it is quite straightforward to detect the place where $\kk$ enters from the construction of the coupling, see Section \ref{Sec:(i) => (iii)}, the function $\kk$ was seemingly never mentioned explicitly in the literature before \cite{braun2019}.) In the Riemannian case, Theorem \ref{Th:Equivalences} establishes a similar result in full generality without any lower boundedness assumption on $k$.  We point out that, in contrast to \cite{braun2019}, the coupling technique on manifolds does not require any notion of ``Wasserstein contractivity'' for the dual heat flow to $(\ChHeat_t)_{t\geq 0}$ on the space of Borel probability measures on $\mms$. It is rather provided in a direct way by the method of \emph{coupling by parallel displacement} \cite{cranston1991, kendall1986}, see \cite{arnaudon2011} for a treatise in the case of constant $k$. Let us also point out \cite{veysseire2011}, which claims the existence of a coupling $(\B^x,\B^y)$ of Brownian motions, possibly with drift, such that for every $t>0$, even
\begin{equation*}
\met\big(\B_t^x,\B_t^y\big) = \rme^{-\int_0^t \kappa\left(\B_r^x,\B_r^y\right)/2\d r}\,\met(x,y)
\end{equation*}
holds on the event that $\smash{\B_r^x}$ and $\smash{\B_r^y}$ do not belong to each other's cut-locus for every $r\in [0,t]$. The real-valued function $\kappa$, the so-called ``coarse curvature'' of $\mms$, is defined outside the diagonal of $\mms\times\mms$ and is slightly larger than $\kk$.

\subsection{Extensions to possible other settings}\label{Extensions}

Apart from well-known geometric and topological applications \cite{bismut1986, bueler, li1994}, recent results for molecular Schrödinger operators \cite{guneysu2019} suggest a detailed study of \emph{weighted} Riemannian manifolds having Kato-type lower bounds on their Bakry--Émery Ricci tensor. In this context, we note that Theorem \ref{Th:First theorem}, Theorem \ref{Cor:Dynkin} and Theorem \ref{Th:Equivalences}  remain valid if for some $\Phi\in C^2(\mms)$, we replace 
\begin{itemize}[leftmargin=1.25cm]
\item $\meas$ by the weighted measure $\rme^{-2\Phi}\,\meas$,
\item $\Delta$ by the drift Laplacian $\Delta - 2\,\langle \nabla \Phi, \nabla\cdot\rangle$,
\item $\Ric$ by the Bakry--Émery Ricci tensor $\Ric + 2\Hess\Phi$,
\item $(\ChHeat_t)_{t\geq 0}$ by the semigroup generated by $\Delta/2 - \langle\nabla \Phi,\nabla\cdot\rangle$, noting that the latter is again essentially self-adjoint  \cite[Section 2.1]{li1992} on $\smash{\Ell^2_{\rme^{-2\Phi}}(\mms)}$, the $\Ell^2$-space w.r.t.~$\rme^{-2\Phi}\,\meas$,
\item $X^x$ by the diffusion generated by the operator $\Delta/2 - \langle \nabla \Phi,\nabla\cdot\rangle$, see e.g.~\cite[Chapter 3]{wang2014} for the particular form of the corresponding stochastic differential equation and the construction of its solution, and
\item $I$ by the weighted index form stated in Remark \ref{Re:index}.
\end{itemize}
Other appropriate changes compared to the non-weighted setting, if needed, will always be indicated in the sequel.

It would also be interesting to study Theorem \ref{Th:First theorem}, Theorem \ref{Cor:Dynkin} and Theorem \ref{Th:Equivalences} in the context of lower bounds on the Bakry--Émery Ricci curvature $\Ric_Z := \Ric + 2\,\nabla Z$ which is associated to a $C^1$-vector field $Z$ on $\mms$ not necessarily of gradient-type.  
See \cite{wang2005,wang2014} and the references therein for a summary of similar statements under different, more geometric conditions. Given appropriate interpretations of the involved analytic objects, see \cite{wang2005,wang2014} for details, some of the results immediately carry over with trivial modifications (for instance, the chain ``(iii) $\Longrightarrow$ (ii) $\Longleftrightarrow$ (i)'' in Theorem \ref{Th:Equivalences}). On the other hand, many of our arguments, e.g.~Theorem \ref{Th:Form FK} below and thus (i) in Theorem \ref{Th:First theorem}, or Theorem \ref{Th:Sc}, are implicitly based on self-adjointness of the semigroup $(\ChHeat_t)_{t\geq 0}$ and the heat flow on $1$-forms. The latter properties lack in this generality, which is why we restricted ourselves to gradient vector fields.

Finally, a further possible (but highly nontrivial) direction of investigation is the case of manifolds with boundary, taking the heat flow with Neumann boundary conditions. See \cite{chen2012, wang2014} and the references therein for an account on diffusion processes on these. The key difficulty in this context will be to take into account the local time of the boundary appropriately. Results that are relevant in this context have been obtained in \cite{airault1975,arnaudon17, deuschel,hsu, ikeda1981, meritet, wang2014}. 

\paragraph*{Acknowledgments} The authors' collaboration arose from a discussion at the \emph{Japa\-nese-German Open Conference on Stochastic Analysis} at the University of Fukuoka, Japan, in September 2019. The authors gratefully acknowledge the warm hospitality of this institution.

The authors thank the anonymous reviewers for their helpful comments and corrections which lead to a significant improvement of the paper's quality.

\section{Preliminaries}\label{Ch:Preliminaries}

For more details on the heat flows on functions and on $1$-forms collected in this chapter, we refer the reader to \cite{davies,grigoryan2009, guneysu2017, hsu2002, rosenberg1997, strichartz1983} and the references therein. For details on their connection with the underlying stochastic processes, see \cite{ikeda1981,malliavin1997,wang2014}. Moreover, all objects and results presented here have counterparts in the weighted case outlined in Section \ref{Extensions}: the heat flow on functions \cite{grigoryan2009}, Brownian motion (or rather the corresponding Ornstein--Uhlenbeck process) \cite{ikeda1981,li1992, wang2014}, and the heat flow on $1$-forms \cite{li1992}.

\paragraph*{Heat flow on functions} The operator $\Delta/2$ is the generator of the  symmetric strongly local, regular Dirichlet form $\scrE\colon\Ell^2(\mms)\to [0,\infty]$ given by
\begin{equation*}
\scrE(f) := \frac{1}{2}\int_\mms \vert \rmd f\vert^2\d\meas\quad\text{if }f\in \Sob^{1,2}(\mms),\quad \scrE(f) := \infty\quad\text{otherwise}.
\end{equation*}
Note that under our standing assumption on $\mms$, $C_\comp^\infty(\mms)$ is dense in the Sobolev space $\Sob^{1,2}(\mms)$ w.r.t.~its natural norm \cite{aubin} -- in other words, $\smash{\Sob^{1,2}_0(\mms)=\Sob^{1,2}(\mms)}$.

The \emph{heat semigroup} or \emph{heat flow} $(\ChHeat_t)_{t\geq 0}$ introduced in the beginning of this article is directly linked to $\scrE$ by spectral calculus and is a strongly continuous, positivity preserving contraction semigroup of linear, self-adjoint operators in $L^2(M)$. Powerful $\Ell^2$-$\Ell^\infty$-regularization properties of the heat flow on relatively compact subsets of $\mms$, an exhaustion procedure and bootstrapping of regularity imply the existence of the so-called \emph{minimal heat kernel} $\smash{\sfp \in \Cont^\infty\big((0,\infty)\times\mms\times\mms; (0,\infty)\big)}$ on $\mms$, the smallest positive fundamental solution to the heat operator $\partial/\partial t - \Delta/2$. It has the property that for every $f\in \Ell^2(\mms)$ and $t>0$, (a version of) $\ChHeat_tf$ can be represented by
\begin{equation*}
\ChHeat_tf(x) := \int_\mms \sfp_t(x,y)\,f(y)\d\meas(y)\forevery{x\in\mms}.
\end{equation*}
Actually, $(\ChHeat_t)_{t\geq 0}$ extends to a contraction semigroup of linear operators from $\Ell^p(\mms)$ into $\Ell^p(\mms)$ for every $p\in [1,\infty]$ which is strongly continuous  if $p<\infty$ and weak$^*$ continuous if $p=\infty$. Moreover, the previous representation formula is still valid for every $p\in [1,\infty]$ and $f\in\Ell^p(\mms)$. For such $f$, the above properties of the heat kernel show that $\ChHeat_\cdot f\in \Cont^\infty\big((0,\infty)\times \mms\big)$ solves the heat equation
\begin{equation}\label{Eq:Functional heat equ}
\frac{\partial}{\partial t} \ChHeat_t f =\frac{1}{2} \Delta \ChHeat_t f\quad\text{in }(0,\infty)\times\mms
\end{equation}
in the classical sense. In addition, we have $\ChHeat_\cdot f\in \Cont^\infty\big([0,\infty)\times \mms\big)$ if $f$ is  also smooth, and $\partial^\alpha\ChHeat_tf$ converges pointwise to $\partial^\alpha f$ as $t\downarrow 0$ for every multiindex $\smash{\alpha\in \N_0^{\dim(\mms)}}$.

\paragraph*{Brownian motion} Given a locally compact Polish space $Y$ we denote by $\Cont([0,\infty);Y)$ the space of continuous maps $\gamma\colon [0,\infty)\to Y$, equipped with the topology of locally uniform convergence and the induced Borel $\sigma$-algebra. Let $Y_\partial := Y\cup\{\partial\}$ denote the one-point compactification of $Y$.

Given a point $x\in\mms$, any stochastic process $\B$ with sample paths in $\Cont([0,\infty); \mms_\partial)$ which is defined on a probability space $(\Omega,\scrF,\bbP)$ (i.e.~the map $t\mapsto \B_t(\omega)$ belongs to $\Cont([0,\infty); \mms_\partial)$ for all $\omega\in\Omega$) is termed \emph{Brownian motion on $\mms$ starting in $x$} if its law equals the Wiener measure $\bbP_x$ on $\Cont([0,\infty);\mms_\partial)$ concentrated at paths starting in $x$. (Usually we want to underline the dependency of $\B$ from its starting point $x$, whence we shall often write $\B^x$.) Recall that $\mathbb{P}_x$ is the uniquely determined probability measure on $\Cont([0,\infty);\mms_\partial)$ with $(\eval_0)_\push\bbP_x = \delta_x$ (where $\eval_0(\gamma) := \gamma_0$ is the evaluation map at $0$) and whose transition density is given by the function $\sfq\colon (0,\infty)\times \mms_\partial\times\mms_\partial \to [0,\infty)$ defined by setting, for every $y,y'\in\mms$, 
\begin{equation*}
\sfq_t(y,y') = \sfp_t(y,y'),\quad \sfq_t(\partial,y') := 0,\quad \sfq_t(\partial,\partial) := 1,\quad \sfq_t(y,\partial) := 1-\int_\mms \sfp_t(y,z)\d\meas(z).
\end{equation*}

Now let $\zeta^x := \inf\{t\geq 0 : \B_t^x = \partial\}$ denote the \emph{explosion time} of $\B^x$, with the usual convention that $\inf\emptyset := \infty$. Since the Wiener measure is concentrated on paths having $\partial$ as a trap, for every $p\in [1,\infty]$ and $f\in L^p(\mms)$ one has
\begin{equation}\label{Eq:Formula Pt BM}
\ChHeat_{t}f(x) = \bbE\big[f(\B_t^x)\,\One_{\{t<\zeta^x\}}\big]\forevery{x\in\mms,\ t\geq 0}.
\end{equation} 
Therefore, $\mms$ is stochastically complete if and only if
\begin{equation}\label{Eq:Equivalent stoch compl}
\bbP\big[t<\zeta^x\big] = \ChHeat_t\One_\mms(x) = \int_\mms\sfp_t(x,y)\d\meas(y) = 1\forevery{x\in\mms,\ t> 0}.
\end{equation}

If $(\Omega,\mathscr{F}_*,\mathbb{P})$ is filtered and $\B^x$ adapted to the given filtration, then $\B^x$ is called an \emph{adapted Brownian motion}. In this case, $\B^x$ is a semimartingale on $\mms$ in the sense that for every $f\in \Cont^\infty(\mms)$, the real-valued process $f\circ \B^x$ is a semimartingale up to the explosion time $\zeta^x$. The \emph{stochastic parallel transport} along $\B^x$ w.r.t.~$\nabla$ started in $x\in\mms$ is a process $\smash{\pa^x}$ constructed in terms of the horizontal lift of $X^x$ to the orthonormal frame bundle over $\mms$, making the linear map $\smash{\pa_t^x\colon T_x\mms \to T_{\B_t^x}\mms}$  a.s.~orthogonal for every $t\in [0,\zeta^x)$.

\paragraph*{Heat flow on {\boldmath{$1$}}-forms} In the sequel, Borel equivalence classes of $1$-forms on $\mms$ having a certain regularity $\scrR$ are denoted by $\Gamma_{\scrR}(T^*\mms)$, and similarly $\Gamma_\scrR(T\mms)$ for Borel equivalence classes of vector fields with regularity $\scrR$. For instance, given $p\in[1,\infty]$, we get the Banach space $\Gamma_{\Ell^p}(T^*\mms)$ given by all Borel equivalence classes $\omega$ of sections in $T^*\mms$ such that $\vert\omega\vert\in \Ell^p(\mms)$. Let $\smash{\vec{\Delta} := \rmd^{\dagger}\rmd + \rmd\d^{\dagger}}$ denote the Hodge--de Rham Laplacian. When defined initially on $\Gamma_{\Cont_\comp^\infty}(T^*\mms)$, by geodesic completeness this operator has a unique self-adjoint extension in the Hilbert space $\Gamma_{\Ell^2}(T^*\mms)$, which will be denoted with the same symbol again. Note our sign convention: $\smash{\vec{\Delta}}$ is nonnegative, while $\Delta$ is nonpositive. The heat semigroup $\smash{(\WHeat_t)_{t\geq 0}}$ on $1$-forms given by $\smash{\WHeat_t := \rme^{-t\vec{\Delta}/2}}$ in $\smash{\Gamma_{\Ell^2}(T^*\mms)}$ is smooth, in the sense for every $\smash{\omega\in \Gamma_{\Ell^2}(T^*\mms)}$ one has a jointly smooth representative $\smash{\WHeat_\cdot\omega}$ which solves the heat equation
\begin{equation*}
\frac{\partial}{\partial t}\WHeat_t \omega = -\frac{1}{2}\vec{\Delta}\WHeat_t\omega\quad\text{in }(0,\infty)\times \mms
\end{equation*}
on $1$-forms with initial condition $\omega$ (and in $[0,\infty)\times \mms$ if $\omega$ is also smooth).

On exact forms, $\smash{\WHeat_t}$ can be represented by the heat operator $\ChHeat_t$; more precisely, for every $f\in W^{1,2}(\mms)$ one has \cite{driver2001,li1992}
\begin{equation}\label{Eq:Commutation}
\WHeat_t\rmd f = \rmd\ChHeat_t f\forevery{t\geq 0}.
\end{equation}
If one drops geodesic completeness of $\mms$, such a commutation relation becomes subtle (cf. \cite{thalmaier} for a negative and \cite{guneysu2017} for a positive result in this direction).

Lastly, a key result is the Feynman-Kac formula, for which we recall the process $\Q^x$ from \eqref{Eq:Q ODE}. Compare with Section \ref{Sec:2}. Note that the last asserted estimate in the theorem follows from Gronwall's inequality, cf.~e.g.~\eqref{Eq:Gronwall} below. See also \cite{EL,malliavin1974} for the compact case.

\begin{citetheorem}[{\cite[Theorem B.4]{driver2001}}]\label{Th:Form FK} Suppose that $\Ric\geq k$ on $\mms$ for some continuous $k\colon \mms \to \R$ satisfying \eqref{Eq:Exp integr assumption}. Then for every $t>0$ and every $\omega\in\Gamma_{L^\infty}(T^*\mms)$ with compact support, the \emph{Feynman--Kac formula}
\begin{equation*}
\WHeat_t\omega(x) = \bbE\big[\Q_t^x\,(\Par_t^x)^{-1}\,\omega^\sharp(\B_t^x)\,\One_{\{t < \zeta^x\}}\big]^\flat\forevery{x\in\mms}
\end{equation*}
holds, and in particular
\begin{equation*}
\big\vert\WHeat_t\omega(x)\big\vert \leq \bbE\Big[\rme^{-\int_0^t k(\B_r^x)/2\d r}\,\big\vert\omega(\B_t^x)\big\vert\,\One_{\{t < \zeta^x\}}\Big]\leq \sfC_t\,\Vert\omega\Vert_{\Ell^\infty}\forevery{x\in\mms},
\end{equation*}
where $\sfC_t$ is defined in \eqref{Eq:Exp integr assumption}.
\end{citetheorem}

%\begin{citetheorem}[{\cite[Theorem B.4]{driver2001}}]\label{Th:Form FK} Let $t>0$ and suppose that $\Ric\geq k$ on $\mms$ for some continuous $k\colon \mms \to \R$. Assume that for every compact $K\subset\mms$, we have
%\begin{equation*}
%\bbE\Big[\rme^{\int_0^t k^-(\B_r^x)/2\d r}\,\One_{\{\B_t^x\in K\}}\,\One_{\{t < \zeta^x\}}\Big] < \infty\forevery{x\in\mms}.
%\end{equation*}
%Then for every $\omega\in\Gamma_{L^2}(T^*\mms)$ with compact support, the \emph{Feynman--Kac formula}
%\begin{equation*}
%\WHeat_t\omega(x) = \bbE\big[\Q_t^x\,(\Par_t^x)^{-1}\,\omega^\sharp(\B_t^x)\,\One_{\{t < \zeta^x\}}\big]^\flat\forevery{x\in\mms}
%\end{equation*}
%holds, and in particular, if in addition $\omega$ is bounded, 
%\begin{equation*}
%\big\vert\WHeat_t\omega(x)\big\vert \leq \bbE\Big[\rme^{-\int_0^t k(\B_r^x)/2\d r}\,\big\vert\omega(\B_t^x)\big\vert\,\One_{\{t < \zeta^x\}}\Big]\leq \sfC_t\,\Vert\omega\Vert_{\Ell^\infty}\forevery{x\in\mms},
%\end{equation*}
%where $\sfC_t$ is defined in \eqref{Eq:Exp integr assumption}.
%\end{citetheorem}

\begin{remark} On weighted Riemannian manifolds, in the notation of Section \ref{Extensions} one has to  replace $(\WHeat_t)_{t\geq 0}$ by the semigroup -- defined on the Hilbert space of $1$-forms that are $L^2$ w.r.t.~$\smash{\rme^{-2\Phi}\,\meas}$ -- which is generated by the essentially self-adjoint operator $\smash{-\vec{\Delta} - 2\d\, i_{\nabla\Phi} - 2\, i_{\nabla\Phi}\,\rmd}$. Here $i_{\nabla\Phi}$ denotes interior multiplication of differential forms with the vector field $\nabla\Phi$ \cite[Section 1.5]{li1992}.
\end{remark}

\section{Proof of Theorem \ref{Th:First theorem} and Theorem \ref{Cor:Dynkin}}

This chapter treats the stochastic completeness of $\mms$, Bismut--Elworthy--Li's derivative formula, and the $\Ell^\infty$-$\Lip$-regularization of the heat semigroup $(\ChHeat_t)_{t\geq 0}$ if we have $\Ric\geq k$ on $\mms$ for some continuous function $k\colon \mms\to\R$ satisfying \eqref{Eq:Exp integr assumption}.

\subsection{Stochastic completeness}\label{Sec:Stochastic completeness}

A key tool for proving stochastic completeness under geodesic completeness, already used in \cite{bakry1986}, are sequences of first-order cutoff-functions \cite[Chapter 2]{strichartz1983}. Their existence is equivalent to the geodesic completeness of $\mms$ (\cite{pigola}, see also \cite{guneysu2016}).

\begin{lemma}\label{Le:Cutoff functions} There exists a sequence $(\psi_n)_{n\in\N}$ in $\Cont_\comp^\infty(\mms)$ satisfying
\begin{enumerate}[leftmargin=1.25cm,label=\textnormal{(\roman*)}]
\item $\psi_n(\mms) \subset [0,1]$ for every $n\in\N$,
\item for all compact $K\subset \mms$, there exists $N\in\N$ such that $\onto{\psi_n}{K} = \One_K$ for every $n\geq N$, and
\item $\smash{\Vert \rmd\psi_n \Vert_{L^{\infty}} \to 0}$ as $n\to\infty$.
\end{enumerate}
\end{lemma}

\begin{proof}[Proof of \textnormal{(i)} in Theorem \ref{Th:First theorem}] We show  \eqref{Eq:Equivalent stoch compl}, i.e.~that $\ChHeat_{t}\One_\mms = \One_{\mms}$ for every $t > 0$. Let $\phi\in\Cont_\comp^\infty(\mms)$, and let $(\psi_n)_{n\in\N}$ be a sequence of first-order cutoff functions provided by Lemma \ref{Le:Cutoff functions}. Then Theorem \ref{Th:Form FK} applied to the $1$-form $\omega := \rmd\psi_n$ for every $n\in\N$ gives
\begin{align*}
\big\Vert\WHeat_s\rmd\psi_n \big\Vert_{\Ell^\infty} \leq %\sup_{x\in\mms} \bbE\Big[\rme^{-\int_0^{s} k(\B_r^x)/2\d r}\,\big\vert\rmd\psi_n(\B_{s}^x)\big\vert\,\One_{\{s<\zeta^x\}}\Big]\\&\leq \sup_{x\in\mms}\bbE\Big[\rme^{\int_0^s k^-(\B_r^x)/2\d r}\,\One_{\{s<\zeta^x\}}\Big]\,\big\Vert\vert\rmd\psi_n\vert\big\Vert_{L^{\infty}}  
\sfC_s\,\Vert \rmd\psi_n \Vert_{L^{\infty}}\leq \sfC_t\,\Vert\rmd\psi_n\Vert_{L^{\infty}},
\end{align*}
uniformly in $s\in [0,t]$. Since $\ChHeat_{\cdot}\psi_n$ solves the heat equation \eqref{Eq:Functional heat equ}, also using Fubini's theorem, integration by parts as well as the commutation rule \eqref{Eq:Commutation} we arrive at
\begin{align*}
\int_\mms \big(\ChHeat_{t}\psi_n - \psi_n\big)\,\phi\d\meas &=\frac{1}{2} \int_\mms \int_0^{t} \phi\,\Delta\ChHeat_s\psi_n\d s\d\meas\\
&= -\frac{1}{2}\int_0^{t}\!\int_\mms \big\langle \rmd \phi,\rmd\ChHeat_s\psi_n\big\rangle\d\meas \d s\\
&= -\frac{1}{2}\int_0^{t}\!\int_\mms \big\langle \rmd \phi, \WHeat_s\rmd\psi_n\big\rangle\d\meas \d s.
\end{align*}
Therefore, we obtain
\begin{align*}
\bigg\vert\!\int_\mms \big(\ChHeat_{t}\One_\mms - \One_\mms\big)\, \phi\d\meas\,\bigg\vert &= \lim_{n\to\infty} \bigg\vert\!\int_\mms \big(\ChHeat_{t}\psi_n - \psi_n\big)\, \phi\d\meas\,\bigg\vert\\
&\leq \limsup_{n\to\infty} \frac{1}{2}\int_0^{t}\!\int_\mms \vert\rmd \phi\vert\,\big\vert\WHeat_{s}\rmd\psi_n\big\vert\d\meas\d s\\
\textcolor{white}{\int_0^t}&\leq \frac{\sfC_t\,t}{2}\,\Vert\rmd \phi\Vert_{L^1}\,\limsup_{n\to\infty}\Vert  \rmd\psi_n\Vert_{L^{\infty}} = 0.
\end{align*}
Since $\phi$ was arbitrary, this proves the claim.
\end{proof}

\subsection[Bismut--Elworthy--Li's derivative formula and the Lipschitz smoothing property]{Bismut--Elworthy--Li's derivative formula and the Lipschitz\\smoothing property}\label{Sec:Bismut}

In view of proving Bismut--Elworthy--Li's derivative formula and the $\Ell^\infty$-$\Lip$-regulariza\-tion property of $(\ChHeat_t)_{t\geq 0}$, for convenience we state the following version of the Burkholder--Davis--Gundy inequality for $q\in [1,\infty)$ (although we only need the upper bounds, respectively), which improves the classically known constants to better ones.

\begin{citelemma}[{\cite[Theorem 2]{ren2008}}]\label{Le:BDG} Let $(\M_r)_{r\geq 0}$ be a real-valued continuous local martingale with $\M_0 = 0$, and let $q\in [1,\infty)$. Then
\begin{equation*}
(8q)^{-q/2}\,\bbE\big[[\M]_\tau^{q/2}\big]\leq \bbE\bigg[\sup_{r\in [0,\tau]}\vert\M_r\vert^q\bigg]\leq (8q)^{q/2} \,\bbE\big[[\M]_\tau^{q/2}\big]
\end{equation*}
for every stopping time $\tau$, where $([\M]_r)_{r\geq 0}$ denotes the quadratic variation process of $(\M_r)_{r\geq 0}$.
\end{citelemma}

Recall the process $\Q^x$ defined by \eqref{Eq:Q ODE} and taking values in $T_x\mms$.

\begin{proof}[Proof of \textnormal{(ii)} in Theorem \ref{Th:First theorem}] Fix $x\in\mms$, $t>0$ and $\xi\in T_x\mms$. It suffices to assume $\smash{\vert\xi\vert \leq 1 }$. We first assume that $f\in\Cont_\comp^\infty(\mms)$. By \cite[Proposition 3.2]{driver2001} and keeping in mind that $\zeta^x = \infty$ a.s., the process $\NN^x$ given by
\begin{equation*}
\NN_r^x := \Big\langle\Q_r^x\,(\Par_r^x)^{-1}\,\nabla\ChHeat_{t-r}f(\B^x_r),\frac{t-r}{t}\,\xi\Big\rangle + \frac{1}{t}\,\ChHeat_{t-r}f(\B_r^x)\int_0^r\!\big\langle\Q_s^x \xi,\rmd\W_s^x\big\rangle,
\end{equation*}
$r\in [0,t]$, is a local martingale. We show that under the given assumption \eqref{Eq:Exp integr assumption} on $k$, this process is even a martingale. 

As already indicated in Theorem \ref{Th:Form FK}, given any $s\geq 0$, it follows from Gronwall's inequality and $\Ric\geq k$ on $\mms$ that a.s.~we have
\begin{equation}\label{Eq:Gronwall}
\big\vert \Q_s^x\big\vert \leq \rme^{-\int_0^s k(\B_r^x)/2\d r} \leq \rme^{\int_0^s k^-(\B_r^x)/2\d r}.
\end{equation}
Hence, for every $q\in [1,\infty)$, by Lemma \ref{Le:BDG} we obtain
\begin{align}\label{Eq:Lq bound Q_s}
\bbE\bigg[\sup_{r\in [0,t]}\bigg\vert\!\int_0^r\!\big\langle\Q_s^x \xi(x),\rmd\W_s^x\big\rangle\bigg\vert^q\bigg] &\leq (8q)^{q/2}\,\bbE\bigg[\bigg(\int_0^t\vert\Q_s^x\vert^2\d s\bigg)^{\!q/2}\bigg]\nonumber\\
&\leq (8q)^{q/2}\,t^{q/2}\, \sup_{y\in\mms}\bbE\Big[\rme^{\int_0^t qk^-(\B_r^y)/2\d r}\Big].
\end{align}
(This estimate will only be needed for $q=1$ in this proof, but is derived for arbitrary $q$ as above for later convenience.) Now, estimating $\vert \Q_r^x\vert$ as in \eqref{Eq:Gronwall} above and using the commutation relation \eqref{Eq:Commutation} as well as Theorem \ref{Th:Form FK}, for all $r\in [0,t]$ one a.s.~has 
\begin{align*}
\vert\NN_r^x\vert &\leq \rme^{\int^{r}_0k^-(\B_s^x)/2 \d s}\, \big\vert \WHeat_{t-r} \d f(\B_r^x)\big\vert+\frac{\Vert f\Vert_{\Ell^\infty}}{t}\,\left\vert \int^{r}_0\!\big\langle\Q_s^x \xi,\d\W_s^x\big\rangle\right\vert\\
%&\leq \rme^{\int^{t}_0k^-(\B_s^x)/2 \d s}\, \bbE_{\B_r^x}\Big[\rme^{-\int^{t-r}_0k(\Y_s)/2 \d s}\,\big\vert\rmd f(\Y_{t-r})\big\vert\Big]+\frac{\Vert f\Vert_{\Ell^\infty}}{t}\,\left\vert \int^{r}_0\!\big\langle\Q_s^x \xi,\rmd\W_s^x\big\rangle\right\vert\\
%&\leq \rme^{\int^{t}_0k^-(\B_s^x)/2 \d s}\, \sup_{y\in\mms}\bbE\Big[\rme^{-\int^{t-r}_0k(\B_s^y)/2 \d s}\Big]\,\big\Vert \vert \rmd f\vert\big\Vert_{\Ell^\infty} + \frac{\Vert f\Vert_{\Ell^\infty}}{t}\,\left\vert \int^{r}_0\!\big\langle\Q_s^x \xi,\rmd\W_s^x\big\rangle\right\vert\\
&\leq \rme^{\int^{t}_0k^-(\B_s^x)/2 \d s}\,\sfC_{t-r}\,\Vert \rmd f\Vert_{\Ell^\infty} + \frac{\Vert f\Vert_{\Ell^\infty}}{t}\,\left\vert\int^{r}_0\!\big\langle\Q_s^x\xi,\rmd\W_s^x\big\rangle\right\vert\\
&\leq \rme^{\int^{t}_0k^-(\B_s^x)/2 \d s}\,\sfC_{t}\,\Vert \rmd f\Vert_{\Ell^\infty} + \frac{\Vert f\Vert_{\Ell^\infty}}{t}\,\left\vert\int^{r}_0\!\big\langle\Q_s^x\xi,\rmd\W_s^x\big\rangle\right\vert.
\end{align*}
%Here, we denote by $\Y$ a generic path in $\mathrm{C}([0,\infty);\mms)$. 
It follows that
\begin{align*}
\bbE\bigg[\sup_{r\in [0,t]}\vert \NN_r^x\vert\bigg]\leq \sfC_t^2\,\Vert \rmd f \Vert_{\Ell^\infty}+\frac{\Vert f\Vert_{\Ell^\infty}}{t}\,\bbE\bigg[\sup_{r\in [0,t]}\bigg\vert\!\int^{r}_0\!\big\langle\Q_s^x\xi,\rmd\W_s^x\big\rangle\bigg\vert\bigg].
\end{align*}
The first summand on the right-hand side is finite thanks to \eqref{Eq:Exp integr assumption}. Estimating the second summand by \eqref{Eq:Lq bound Q_s} above for $q=1$, also the second summand is finite again by \eqref{Eq:Exp integr assumption}. It follows that $\NN^x$ is a true martingale, and thus
\begin{equation}\label{Eq:Bismut proof}
\big\langle \nabla\ChHeat_tf(x),\xi \big\rangle= \bbE\big[\NN_0^x\big]=\bbE\big[\NN_t^x\big]=\frac{1}{t}\,\bbE\left[ f(\B_t^x)\int^{t}_0\!\big\langle\Q_s^x \xi,\rmd\W_s^x\big\rangle\right].
\end{equation}

The claimed equality for bounded $f\in C^\infty(\mms)$ follows by replacing $f$ by $\psi_n\,f$ in \eqref{Eq:Bismut proof} for every $n\in\N$, where $(\psi_n)_{n\in\N}$ is as in Lemma \ref{Le:Cutoff functions}, and letting $n\to \infty$ (together with the dominated convergence theorem on the right-hand side). In turn, if only $f\in L^\infty(\mms)$, a similar procedure works by replacing $f$ by $\ChHeat_\varepsilon f$ in \eqref{Eq:Bismut proof}, where $\varepsilon > 0$, and letting $\varepsilon \to 0$.
\end{proof}

\begin{proof}[Proof of \textnormal{(iii)} in Theorem \ref{Th:First theorem}] Using the previous Bismut--Elworthy--Li formula  and \eqref{Eq:Lq bound Q_s} above for $q=1$, for every $x\in M$, $t>0$ and $\xi\in T_xM$ with $\vert\xi\vert\leq 1$, we get 
\begin{align*}
\big\vert\big\langle\nabla \ChHeat_tf(x),\xi\big\rangle\big\vert &\leq \frac{1}{t}\,\bbE\left[\left\vert\int^t_0\!\big\langle\Q^x_s \xi,\d\W^x_s\big\rangle\right\vert\right]\,\Vert f\Vert_{\Ell^\infty}\\
&\leq \sqrt{8}\,t^{-1/2}\,\sup_{x\in\mms}\bbE\Big[\rme^{\int^t_0 k^-(\B^x_r)/2 \d r}\Big]\,\Vert f\Vert_{\Ell^\infty},
\end{align*} 
and duality gives
\begin{equation*}
\Lip(\ChHeat_tf)\leq \sqrt{8}\, t^{-1/2}\,\sup_{x\in\mms}\bbE\Big[\rme^{\int^t_0 k^-(\B^x_r)/2 \d r}\Big]\,\Vert f\Vert_{\Ell^\infty}.\qedhere
\end{equation*}
\end{proof}

Now we assume Kato decomposability of $k$ in the rest of this chapter, devoting ourselves to the proof of Theorem \ref{Cor:Dynkin}. In this situation, one has to guarantee that the right-hand side of Bismut--Elworthy--Li's formula is  well-defined for $f\in \Ell^p(\mms)$, where $p\in (1,\infty)$, which is essentially the content of the following lemma.

\begin{lemma}\label{Le:Linear operator} Let $t\geq  0$ and $V\in\Gamma_{\Ell^\infty}(T\mms)$. Then for every $f\in\Ell^\infty(\mms)$ and $x\in\mms$, the random variable $\smash{f(\B_t^x)\int_0^t\big\langle\Q_s^xV(x),\rmd\W_s^x\big\rangle}$ is integrable. Moreover, for every $p\in (1,\infty]$, the operator $\sfE_t^V\!$ given on functions $f\in\Ell^\infty(\mms)\cap\Ell^p(\mms)$ in terms of
\begin{equation*}
\sfE_t^V\!f(x):=\bbE\bigg[f(\B_t^x)\int_0^t\!\big\langle\Q_s^xV(x),\rmd\W_s^x\big\rangle\bigg]\forevery{x\in\mms}
\end{equation*}
extends to a bounded linear operator from $\Ell^p(\mms)$ into $\Ell^p(\mms)$, and the previous representation is valid and well-defined for every $f\in\Ell^p(\mms)$.
\end{lemma}

\begin{proof} Let $V\in \Gamma_{\Ell^\infty}(T\mms)$ and $f\in\Ell^\infty(\mms)$, for which we assume without loss of generality that $\smash{\big\Vert \vert V\vert\big\Vert_{\Ell^\infty} \leq 1}$ and $\Vert f\Vert_{\Ell^\infty} \leq 1$. The inequality \eqref{Eq:Lq bound Q_s} for $q=1$ and Lemma \ref{Le:Khas} directly show the claimed integrability of  $\smash{f(\B_t^x)\int_0^t\big\langle\Q_s^xV(x),\d\W_s^x\big\rangle}$, and they also show that $\smash{\sfE_t^V}\!$ is a bounded linear operator from $\Ell^\infty(\mms)$ into $\Ell^\infty(\mms)$. 

%If $p\in (1,\infty)$, then Hölder's inequality, \eqref{Eq:Lq bound Q_s} for $q = p/(p-1)$, contractivity of $(\ChHeat_t)_{t\geq 0}$ and Theorem \ref{Cor:Lp properties} below show that there exist finite constants $C_1,C_2\geq 0$ depending only on $k^-$ and $p$ such that for every $f\in \Ell^p(\mms)\cap\Ell^\infty(\mms)$,
%\begin{equation*}
%\big\Vert \sfE_t^V\!f\big\Vert_{\Ell^p} \leq C_1\,t^{(p-1)/2p}\,\rme^{C_2 t}\,\Vert f\Vert_{\Ell^p},
%\end{equation*}
%and we conclude by an approximation argument as in the proof of Theorem \ref{Cor:Lp properties}.

If $p\in (1,\infty)$, successively using Hölder's inequality, \eqref{Eq:Lq bound Q_s} for $q = p/(p-1)$, Lemma \ref{Le:Khas} again and mass preservation of $(\ChHeat_t)_{t\geq 0}$, we infer the existence of a finite constant $\mathsf{C}> 0$ depending only on $k^-$, $t$ and $p$ such that for every $f\in \Ell^p(\mms)\cap\Ell^\infty(\mms)$,
\begin{align*}
	\big\Vert \sfE_t^V\!f\big\Vert_{\Ell^p}^p &= \int_{\mms} \Big\vert\bbE\Big[f(\B_t^x)\int_0^t\!\big\langle\Q_s^xV(x),\rmd\W_s^x\big\rangle\Big]\Big\vert^p\d\meas(x)\\
	&\leq \int_{\mms} \bbE\big[\vert f(\B_t^x)\vert^p\big]\,\bbE\Big[\Big\vert\!\int_0^t\!\big\langle\Q_s^xV(x),\rmd\W_s^x\big\rangle\Big\vert^q\Big]^{p/q}\d\meas(x)\\
	&\leq (8q)^{p/2}\,t^{p/2}\,\sup_{y\in\mms}\bbE\Big[\rme^{\int_0^t qk^-(\B_r^x)/2\d r}\Big]^{p-1} \int_{\mms}\ChHeat_t\big(\vert f\vert^p\big)(x)\d\meas(x)\\
	&\leq \mathsf{C}\,(8q)^{p/2}\,t^{p/2}\,\Vert f\Vert^p_{\Ell^p}.
\end{align*}
We conclude the statement by a standard approximation argument.
\end{proof}

\begin{proof}[Proof of Theorem \ref{Cor:Dynkin}] Trivially, $L^\infty(\mms)\cap L^p(\mms)$ is dense in $L^p(\mms)$. Note that, given $p \in (1,\infty)$, and $f\in\Ell^p(\mms)$, it follows from the divergence theorem as well as Lemma \ref{Le:Linear operator} -- replacing $\xi$ by an appropriate smooth and bounded vector field $V\in \Gamma(T\mms)$ such that $V(x)= \xi$ -- that both sides of \eqref{Eq:Bismut proof} are continuous in $f$ w.r.t.~convergence in $\Ell^p(\mms)$. In particular, the desired pointwise identity follows.
\end{proof}

\section{Proof of Theorem \ref{Th:Equivalences}}

We turn to characterizations of continuous lower Ricci curvature bounds in terms of functional inequalities and existence of couplings. Throughout this chapter, we assume that $k\colon \mms\to\R$ is continuous, and only state explicitly if we need \eqref{Eq:Exp integr assumption}.

\subsection[From the $\Ell^1$-Bochner inequality to lower Ricci bounds]{From the {\boldmath{$\Ell^1$}}-Bochner inequality to lower Ricci bounds}\label{Sec:3}

As already hinted, the key point in showing the implication ``(ii) $\Longrightarrow$ (i)'' in Theorem \ref{Th:Equivalences} is the well-known Bochner formula \eqref{Eq:Bochner formula}, subject to a clever choice of $f$ as granted by the subsequent lemma, together with the chain rule to deduce $\Ric\geq k$ on $\mms$. 

It is well-known in Riemannian geometry that, given any $x\in\mms$, there exists an open subset $O_x \subset T_x\mms$ such that the restriction of the exponential map to $O_x$ provides a diffeomorphism $\exp_x \colon O_x \to \exp_x(O_x)$. We denote its inverse by $\exp_x^{-1}$.

\begin{citelemma}[{\cite[Lemma 3.2]{vonrenesse2005}}]\label{Le:Nice function} Let $x\in\mms$ and $\xi\in T_x\mms$ with unit norm. Let $\scrH := \big\lbrace\!\exp_x \eta : \eta \in O_x,\ \langle \eta,\xi\rangle = 0\big\rbrace$ be the $(\dim(\mms)-1)$-dimensional hypersurface in $\mms$ orthogonal to $\xi$ at $x$. Then there exists an open neighborhood $U \subset \exp_x(O_x)$ of $x$ such that the \emph{signed distance function} $\smash{\sdist_\scrH\colon U \to \R}$ given by
\begin{equation*}
\sdist_\scrH(y) := \rho(y,\scrH)\,\sgn\!\big\langle\xi,\exp_x^{-1}y\big\rangle,\quad\text{where}\quad \rho(y,\scrH) := \inf_{z\in\scrH}\rho(y,z),
\end{equation*}
obeys
\begin{equation*}
\sdist_\scrH \in \Cont^\infty(U),\quad \nabla \sdist_\scrH(x) = \xi,\quad \big\vert \nabla \sdist_\scrH(U)\big\vert = \{1\},\quad\Hess\sdist_\scrH(x) = 0.
\end{equation*}
\end{citelemma}

\begin{proof}[Proof of ``\textnormal{(ii)} $\Longrightarrow$ \textnormal{(i)}'' in Theorem \ref{Th:Equivalences}] Let $x\in \mms$, and let $\xi\in T_x\mms$ obey $\vert\xi\vert = 1$. In the notation from Lemma \ref{Le:Nice function}, consider the function $\smash{f:=\rho_\scrH^\pm}$ provided therein. By Lemma \ref{Le:Nice function}, Bochner's formula \eqref{Eq:Bochner formula} and the chain rule for $\Delta$, we have
\begin{align*}
\Ric(x)(\xi,\xi) &= \Delta\frac{\vert\nabla f\vert^2(x)}{2} -\big\langle\nabla\Delta f(x),\nabla f(x)\big\rangle\\
&= \vert\nabla f\vert\,\Delta\big\vert\nabla f(x)\vert + \big\vert\nabla\vert\nabla f\vert(x)\big\vert^2 -\big\langle\nabla\Delta f(x),\nabla f(x)\big\rangle\\
&\geq k(x)\,\vert\nabla f(x)\vert^2 = k(x).\textcolor{white}{\frac{1}{2}}
\end{align*}
The arbitrariness of $\xi$ concludes the proof.
\end{proof}

\begin{remark} In the weighted setting outlined in Section \ref{Extensions} -- retaining the notation therein -- the only essential change needed to modify the previous proof is to replace the unweighted Bochner identity \eqref{Eq:Bochner formula} by its weighted counterpart
	\begin{align*}
		\Delta^\Phi \frac{\vert\nabla f\vert^2}{2} = \langle\nabla\Delta^\Phi f,\nabla f\rangle + \vert\!\Hess f\vert^2 + \Ric^\Phi(\nabla f,\nabla f),
	\end{align*}
	where $\Delta^\Phi := \Delta-2\,\langle\nabla\Phi,\nabla\cdot\rangle$ and $\Ric^\Phi := \Ric + 2\Hess\Phi$. The latter follows from \eqref{Eq:Bochner formula}, the definition of $\Hess \Phi$ and metric compatibility of $\nabla$, see e.g.~page 28 in \cite{petersen}:
	\begin{align*}
		2\Hess\Phi(\nabla f,\nabla f) &= 2\,\big\langle \nabla_{\nabla f}\nabla\Phi,\nabla f\big\rangle\\
		&= 2\,\big\langle\nabla \langle\nabla\Phi,\nabla f\rangle,\nabla f\big\rangle - \big\langle\nabla\Phi,\nabla\vert\nabla f\vert^2\big\rangle.
	\end{align*}
	The chain rule for $\Delta^\Phi$ is analogous to the one for $\Delta$.
\end{remark}

\subsection{From lower Ricci bounds to pathwise couplings}\label{Sec:(i) => (iii)}

We start with the existence of a suitable coupling of Brownian motions under the inequality $\Ric\geq k$ on $\mms$, also assuming \eqref{Eq:Exp integr assumption} in this section. (Note that the stochastic completeness of $\mms$ is already known by Theorem \ref{Th:First theorem}.) The coupling technique is well-known and called \emph{coupling by parallel displacement}, see \cite{cranston1991, kendall1986, wang2005, wang2014} and the references therein.  See also \cite{wang1994} for a ``local'' treatise on regular subdomains.

We first collect some notation. Denote by $\Cut_v$ the cut-locus of $v\in\mms$, by $\mathrm{diag}$ the diagonal of $\mms\times\mms$, and by $\rmR$ the Riemannian curvature tensor of $\mms$. Abbreviate $d := \dim(\mms)$ and define $\smash{\Cut := \big\lbrace (u,v) \in \mms\times\mms : u\in \Cut_v\big\rbrace}$. Given any $(u,v)\in (\mms\times\mms)\setminus(\mathrm{diag}\cup\Cut)$, let $J_1,\dots,J_{d-1}$ be Jacobi fields along the unique minimal geodesic $\gamma\colon [0,\met(u,v)] \to \mms$ from $u$ to $v$ such that $\{J_1(s),\dots,J_{d-1}(s),\dot{\gamma}_s\}$ is an orthonormal basis of $T_{\gamma_s}\mms$ both for $s=0$ as well as $s=\met(u,v)$. Define the \emph{index form} by
\begin{equation*}
I(u,v) := \sum_{i=1}^{d-1}\int_0^{\met(u,v)} \Big(\big\vert\nabla_{\dot{\gamma}_s}J_i(s)\big\vert^2 - \big\langle \rmR(\dot{\gamma}_s, J_i(s))\dot{\gamma}_s, J_i(s)\big\rangle\Big)\d s.
\end{equation*}

\begin{theorem}\label{Thm:Cranston} For every $x,y\in\mms$ with $x\neq y$, there exists a coupling $(\B^x,\B^y)$ of Brownian motions on $\mms$ starting in $(x,y)$ which coincide past their \emph{coupling time}
	\begin{align*}
	T(\B^x,\B^y) := \inf\!\big\lbrace t\geq 0 : X_t^x = X_t^y\big\rbrace
	\end{align*}
such that for every $I'\in \Cont(\mms\times\mms)$ for which $I'\geq I$ holds outside $\mathrm{diag}\cup \Cut$, before $T(\B^x,\B^y)$ we have
\begin{equation*}
\rmd \met\big(\B_t^x,\B_t^y) \leq \frac{1}{2}\,I'\big(\B_t^x,\B_t^y\big)\d t.
\end{equation*}
\end{theorem}

The construction of this coupling is thoroughly carried out in Theorem 2.1.1 and (the proof of) Proposition 2.5.1 in \cite{wang2005}, see also \cite[Theorem 2.3.2]{wang2014}. The key to deduce ``\textnormal{(i)} $\Longrightarrow$ \textnormal{(iii)}'' in Theorem \ref{Th:Equivalences} from Theorem \ref{Thm:Cranston} now is to construct an appropriate function $I'\in\Cont(\mms\times\mms)$ with $I'\geq I$ outside $\mathrm{diag}\cup\Cut$, hence circumventing cut-locus issues. This is the place where the definition of $\kk$ enters.

\begin{proof}[Proof of ``\textnormal{(i)} $\Longrightarrow$ \textnormal{(iii)}'' in Theorem \ref{Th:Equivalences}] Let $u,v\in (\mms\times\mms)\setminus (\mathrm{diag}\cup \Cut)$.  As in the proof of \cite[Theorem 2.1.4]{wang2005}, let $U_1,\dots,U_{d-1}$ be parallel vector fields along $\gamma$ such that $\{U_1(s),\dots,U_{d-1}(s),\dot{\gamma}_s\}$ is an orthonormal basis of $\smash{T_{\gamma_s}\mms}$ for every $s\in [0,\met(u,v)]$. By the index lemma \cite[Lemma 1.21]{cheeger1975},  we have
\begin{align}\label{Eq:Blu}
\begin{split}
I(u,v) &\leq -\int_0^{\met(u,v)} \bigg[\sum_{i=1}^{d-1} \big\langle \rmR(\dot{\gamma}_s, U_i(s))\dot{\gamma}_s, U_i(s)\big\rangle \bigg]\d s = -\int_0^{\met(u,v)} \Ric(\gamma_s)(\dot{\gamma}_s,\dot{\gamma}_s)\d s\\
&\leq -\int_0^{\met(u,v)} k(\gamma_s)\d s \leq -\met(u,v)\,\kk(u,v).
\end{split}
\end{align}
As $\kk$ is lower-semicontinuous, a well-known consequence of Baire's theorem yields the existence of a pointwise increasing sequence $(\kk_n)_{n\in\N}$ in $\Cont(\mms\times\mms)$ converging pointwise to $\kk$. Applying Theorem \ref{Thm:Cranston} with $I'$ replaced by $I'_n\in \Cont(\mms\times\mms)$ given by $I'_n(u,v) := -\met(u,v)\,\kk_n(u,v)$ and integrating the resulting differential inequality, a.s.~we have
\begin{align*}
	\met\big(X_t^x,X_t^y\big) \leq \rme^{-\int_s^t \kk_n\left(\B_r^x,\B_r^y\right)/2\d r}\,\met\big(\B_s^x,\B_s^y\big)
\end{align*}
for every $s,t\geq 0$ with $s \leq t$ and for every $n\in\N$. (Recall that $\B^x$ and $\B^y$ coincide past their coupling time.) Letting $n\to\infty$ with the aid of the monotone convergence theorem, we obtain the desired pathwise estimate.
\end{proof}

\begin{remark}\label{Re:index} In the weighted case from Section \ref{Extensions}, the quantity $I$ has to be replaced by its weighted counterpart
	\begin{align*}
	I_\Phi(u,v) := I(u,v) -(\nabla\Phi)\met(\cdot,v)(u) -(\nabla \Phi)\met(\cdot,u)(v).
	\end{align*}
	Theorem \ref{Thm:Cranston} remains true for the corresponding diffusion process, and the weighted adaptation of the estimates \eqref{Eq:Blu} follows the proof of \cite[Theorem 2.1.4]{wang2005}.
\end{remark}

\subsection[From pathwise couplings to the $\Ell^1$-Bochner inequality]{From pathwise couplings to the {\boldmath{$\Ell^1$}}-Bochner inequality}\label{Sec:(iii) => (ii)}

Even if $k$ is smooth, as implicitly discovered in Section \ref{Sec:(i) => (iii)} above, the function $\kk$ from \eqref{Eq:kk definition} is in general only lower semicontinuous. In the current section, we shall need to bypass this lack of continuity by approximation through \emph{Lipschitz} functions. To this aim, the following fact, in which Lipschitz continuity on $\mms\times\mms$ is understood w.r.t.~the product metric $\met_2$ given by $\smash{\met_2^2\big((x,y),(x',y')\big) := \met^2(x,x') + \met^2(y,y')}$, is helpful.

\begin{lemma}\label{Le:Approximation lemma} Let $D\subset\mms$ be a compact subset. Then, in $D\times D$, $\kk$ is the pointwise limit of a pointwise increasing sequence of functions in $\Lip_\bounded(\mms\times\mms)$ which are everywhere not smaller than $\inf \kk(D\times D)$.
\end{lemma}

A purely metric construction of the claimed sequence can be found e.g.~at page 107 in \cite{ags05}, replacing $(X,d)$ by $(D\times D,\met_2)$ therein. In particular, note that since $k(x) = \kk(x,x)$ for every $x\in\mms$, in $D$, $k$ is the pointwise limit of a pointwise increasing sequence of functions in $\Lip_\bounded(\mms)$ which are everywhere not smaller than $\inf k(D)$.

%\begin{proof} By a purely metric argument, every lower semicontinuous, lower bounded function on $\mms\times\mms$ can be approximated pointwise on $\mms\times\mms$ by a pointwise increasing sequence of functions in $\Lip_\bounded(\mms\times\mms)$ which preserves uniform lower bounds, see e.g.~page 107 in \cite{ags05}. If $\kk$ is not uniformly bounded from below, we apply the previous result to the function $\underline{\ell}\colon \mms\times\mms\to\R$ given by $\underline{\ell}(x,y) := \kk(x,y)\,\One_{D\times D}(x,y) + \inf\kk(D\times D)\, \One_{(D \times D)^\rmc}(x,y)$.
%
%The second statement follows  by noting that $k(x) = \kk(x,x)$ for every $x\in\mms$.
%\end{proof}

The step from the pathwise coupling property w.r.t.~$k$ towards \eqref{Eq:1-Bochner} requires a nontrivial extension of the arguments for \cite[Theorem 5.17]{braun2019} (which adapt the duality argument from \cite{kuwada2010} to the case of synthetic variable Ricci bounds and make crucial use of uniform lower boundedness of the Ricci curvature) for short times instead of fixed ones. This kind of localization argument was indeed used in \cite{braun2019} in different variants at different instances. For this, a certain short-time behavior of Brownian motion as subsequently recorded plays a crucial role.

Given any $x\in\mms$ and $\varepsilon> 0$, let $\smash{\tau^x_\varepsilon}$ be the first exit time of Brownian motion starting in a fixed $x\in\mms$ from $\Ball_\varepsilon(x)$. The following estimate for $\smash{\tau^x_\varepsilon}$ is a variant of \cite[Lemma 2.1.4]{wang2014},  noting that by Laplacian comparison, compare with \cite[Corollary 3.4.4, Corollary 3.4.5, Theorem 3.6.1]{hsu2002}, the constant $c_1$ therein can be chosen uniformly in $x$ off its respective cut-locus as long as long as $x$ belongs to a compact subset of $\mms$. (An analogous version of Lemma \ref{Le:Hsu lemma} holds for general gradient diffusions, taking -- in the notation of Section \ref{Extensions} -- into account the continuity of $\nabla\Phi$.)

\begin{lemma}\label{Le:Hsu lemma} For every compact $D\subset\mms$ and every $\varepsilon>0$, there exists a constant $c>0$ such that 
\begin{equation*}
\bbP\big[\tau^x_\varepsilon \leq t\big] \leq \rme^{-c/t}\forevery{x\in D,\ t \in (0,1]}.
\end{equation*}
\end{lemma}

%\begin{proof} We only sketch the proof. Let $D_{2\varepsilon}$ be the $2\varepsilon$-neighborhood of $D$. By Laplacian comparison \cite[Theorem 1.2.3]{wang2014}, there exist two constants $C_1,C_2 > 0$ depending only on $\dim(\mms)$ and the Ricci curvature of $\smash{\overline{D}_{2\varepsilon}}$ such that $\Delta\rho^2(x,\cdot) \leq C_1 + C_2\,\rho^2(x,\cdot) \leq C_1 + C_2\,(\mathrm{diam}\, \overline{D}_{2\varepsilon})^2$ on $D_a\setminus \Cut_x$ for every $x\in D$. Now, we proceed as for \cite[Lemma 2.1.4]{wang2014}, keeping in mind the interpretation of the Laplacian term in Kendall's Itô formula \cite[Theorem 1.1, Corollary 1.2]{kendall1987} for the radial part of Brownian motion as zero at times when the latter hits the cut-locus of its starting point.
%\end{proof}

\begin{proof}[Proof of ``\textnormal{(iii)} $\Longrightarrow$ \textnormal{(ii)}'' in Theorem \ref{Th:Equivalences}] \textbf{Step 1.} \textit{Initial preparations.} Let $\smash{f\in\Cont_\comp^\infty(\mms)}$ and $x\in\mms$ with $\vert \nabla f(x)\vert \neq 0$ be arbitrary. Let $\varepsilon \in (0,1/4]$, which is kept fixed throughout this proof, be such that $\vert\nabla f\vert$ is bounded away from zero -- in particular smooth -- on $\smash{\overline{\Ball}_{4\varepsilon}(x)}$. Moreover, let $\gamma$  be the unique geodesic starting in $x$ with initial velocity $\dot{\gamma}_0 = \nabla f(x)/\vert\nabla f(x)\vert$. The continuity of $k$ yields $k\geq K$ on $\smash{\overline{\Ball}_{6}(x)}$ for some negative real number $K$. Define the set of  points in $\mms$ with distance at most $1$ to $\gamma$ by
\begin{equation*}
D := \bigcup_{s\in [0,1]} \overline{\Ball}_1(\gamma_s).
\end{equation*}
% \big\lbrace z\in \mms : z \in \overline{\Ball}_1(\gamma_s)\text{ for some }s\in [0,1]\big\rbrace$ 
By the definition \eqref{Eq:kk definition} of $\kk$ and since $\met(x,\gamma_s) \leq 1$ for every $s\in [0,1]$, we have
\begin{equation}\label{Eq:kk bound}
\kk\geq K\quad\text{on } D\times D.
\end{equation}
Finally, let $\kkk\in\Lip_\bounded(\mms\times\mms)$ be any function obeying $K\leq \kkk \leq \kk$ on $D\times D$ as provided by Lemma \ref{Le:Approximation lemma}.

\textbf{Step 2.} \textit{Rewriting the quantities to consider.} The key idea to derive the $\Ell^1$-Bochner inequality \eqref{Eq:1-Bochner} for $f$ from the given pathwise coupling estimates is to consider certain difference quotients of the map $(t,s)\mapsto\ChHeat_tf(\gamma_s)$ near $(0,0)$, and to express the involved heat semigroups in terms of coupled Brownian motions. To address the first point, we note that, given $t>0$, by the smoothness of  $s\mapsto \ChHeat_tf(\gamma_s)$ on $[0,\infty)$, Taylor's theorem in its mean value remainder form and the geodesic equation for $\gamma$, given any $s>0$ there exists $\nu\in [0,s]$ such that
\begin{align*}
\ChHeat_t f(\gamma_s) - \ChHeat_t f(x) = s\,\big\langle \nabla \ChHeat_t f(x),\dot{\gamma}_0\big\rangle + \frac{s^2}{2}\Hess\ChHeat_tf(\gamma_\nu)\big(\dot{\gamma}_\nu,\dot{\gamma}_\nu\big).
\end{align*}
Dividing by $s$, subtracting $\big\langle\nabla f(x),\dot{\gamma}_0\big\rangle$ and dividing by $t$, respectively, yields
\begin{align}\label{Eq:Important limit justification}
\begin{split}
&\frac{1}{t}\,\Big[\,\frac{1}{s}\,\big[\ChHeat_tf(\gamma_s) - \ChHeat_tf(x)\big] - \big\langle\nabla f(x),\dot{\gamma}_0\big\rangle\Big]\\
&\qquad\qquad = \frac{1}{t}\,\big\langle \nabla \ChHeat_tf(x) - \nabla f(x),\dot{\gamma}_0\big\rangle + \frac{s}{2t}\Hess\ChHeat_tf(\gamma_\nu)\big(\dot{\gamma}_\nu,\dot{\gamma}_\nu\big).
\end{split}
\end{align}
%We shall also keep in mind that $s = \met(x,\gamma_s)$ for every $s \in (0,\mathrm{inj}_x)$, where $\mathrm{inj}_x$ is the injectivity radius of $x$.

To now invoke the coupled Brownian motions, given $s\in (0,1]$ let us denote by $(\B^x,\B^{\gamma_s})$ a process starting in $(x,\gamma_s)$ given by the pathwise coupling property w.r.t.~$k$. %This pair process still depends on $s$, but we suppress this dependency from the notation. 
Let $\smash{\tau^x_{\varepsilon}}$ and $\smash{\tau^{\gamma_s}_\varepsilon}$ denote the first exit times of the marginal Brownian motions $\B^x$ and $\B^{\gamma_s}$ from $\Ball_\varepsilon(x)$ and $\Ball_\varepsilon(\gamma_s)$, respectively. Since $\overline{\Ball}_\varepsilon(x), \overline{\Ball}_\varepsilon(\gamma_s)\subset D$, for every $s\in [0,1]$ a.s.~on the event $\big\lbrace \tau^x_\varepsilon > t\text{ and } \tau^{\gamma_s}_\varepsilon>t \big\rbrace$ we have
\begin{align}\label{Eq:Pathwise bound maximum}
\met\big(\B_t^x,\B_t^{\gamma_s}\big) \leq \rme^{-\int_0^t\kk\left(\B_r^x,\B_r^{\gamma_s}\right)/2\d r}\,s \leq \rme^{-\int_0^t\kkk\left(\B_r^x,\B_r^{\gamma_s}\right)/2\d r}\,s.
\end{align}
\textbf{Step 3.} \textit{Estimating \eqref{Eq:Important limit justification} via coupled Brownian motions.} Given $\smash{s\in \big(0,\rme^{-1/2}\big]}$, define $t_s := -c/\!\log s^2 \in (0,c]$, where $c>0$ is the constant from Lemma \ref{Le:Hsu lemma} associated to $D$ and $\varepsilon$. Note that $t_s \to 0$ and $s^\alpha/t_s\to 0$ as $s\to 0$ for $\alpha\in \{1/2,1\}$. Consider the events
\begin{align}\label{Eq:Events def}
\begin{split}
A_s &:= \big\lbrace\tau^x_\varepsilon > t_s\text{ and } \tau^{\gamma_s}_\varepsilon > t_s\big\rbrace,\\
V_s &:= A_s \cap \big\lbrace \met\big(\B_{t_s}^x,\B_{t_s}^{\gamma_s}\big)\geq s^{1/2}\big\rbrace,\rule{0cm}{0.6cm}\\
W_s &:= A_s \cap \Big\lbrace \frac{1}{t_s}\int_0^{t_s}\met\big(\B_r^x,\B_r^{\gamma_s}\big)\d r\geq s^{1/2}\Big\rbrace,\rule{0cm}{0.6cm}\\
U_s &:= A_s \cap V_s^\rmc \cap W_s^\rmc.\rule{0cm}{0.45cm}
\end{split}
\end{align}
Since $(t,s) \mapsto \Hess\ChHeat_tf(\gamma_s)(\dot{\gamma}_s,\dot{\gamma}_s)$ is locally bounded on $[0,\infty)\times[0,\infty)$ by joint smoothness of the heat semigroup, by \eqref{Eq:Important limit justification} with $t_s$ in place of $t$ we have
\begin{align*}\label{Eq:Joint smoothness}
&\frac{1}{2}\,\vert\nabla f(x)\vert^{-1}\,\big\langle\nabla \Delta f(x),\nabla f(x)\big\rangle\\
&\qquad\qquad = \lim_{s\downarrow 0} \frac{1}{t_s}\,\big\langle\nabla\ChHeat_{t_s}f(x) - \nabla f(x),\dot{\gamma}_0\big\rangle\\
&\qquad\qquad \leq \limsup_{s\to 0} \frac{1}{t_s}\,\Big[\,\frac{1}{s}\,\big[\ChHeat_{t_s}f(\gamma_s) - \ChHeat_{t_s}f(x)\big] - \big\langle\nabla f(x),\dot{\gamma}_0\big\rangle\Big]\\
&\qquad\qquad \leq \limsup_{s\downarrow 0} \frac{1}{t_s}\, \Big[\,\frac{1}{s}\, \bbE\big[\big\vert f(\B_{t_s}^x) - f(\B_{t_s}^{\gamma_s})\big\vert\big] - \vert\nabla f(x)\vert\Big]\\
&\qquad\qquad = \limsup_{s\downarrow 0} \frac{1}{t_s}\, \Big[\,\frac{1}{s}\, \bbE\big[\big\vert f(\B_{t_s}^x) - f(\B_{t_s}^{\gamma_s})\big\vert\,\big(\One_{V_s} + \One_{W_s} + \One_{U_s} + \One_{A_s^\rmc}\big)\big] - \vert\nabla f(x)\vert\Big].
\end{align*}
Now we estimate the contributions of the events defined in \eqref{Eq:Events def} separately. 

\textbf{Step 3.1.} The contribution of $A_s^\rmc$ becomes negligible thanks to
\begin{align*}
&\limsup_{s\downarrow 0} \frac{1}{t_s\,s}\,
\bbE\big[\big\vert f(\B_{t_s}^x) - f(\B_{t_s}^{\gamma_s})\big\vert\,\One_{A_s^\rmc}\big]\\
&\qquad\qquad\leq \limsup_{s\downarrow 0}\frac{2\, \Vert f\Vert_{\Ell^\infty}}{t_s\,s}\,\bbP\big[A_s^\rmc\big]\\
&\qquad\qquad \leq \limsup_{s\downarrow 0}\frac{2\,\left\|f\right\|_{L^{\infty}}}{t_s\,s}\,\big[\bbP\big[\tau^x\leq t_s\big] + \bbP\big[\tau^{\gamma_s} \leq t_s\big]\big]\\
&\qquad\qquad \leq \limsup_{s\downarrow 0} \frac{4\, \Vert f\Vert_{\Ell^\infty}}{t_s\,s}\, \rme^{-c/t_s} = \limsup_{s\downarrow 0} \frac{4\,\Vert f\Vert_{\Ell^\infty}}{t_s}\,s =0,
\end{align*}
where the last inequality is granted by Lemma \ref{Le:Hsu lemma}. 

\textbf{Step 3.2.} Furthermore, by \eqref{Eq:Pathwise bound maximum} and \eqref{Eq:kk bound}, the contribution of $V_s$ is controlled by
\begin{align*}
&\limsup_{s\downarrow 0} \frac{1}{t_s\,s}\,\bbE\big[\big\vert f(\B_{t_s}^x) - f(\B_{t_s}^{\gamma_s})\big\vert\,\One_{V_s}\big]\\
&\qquad\qquad = \limsup_{s\downarrow 0} \frac{1}{t_s\,s}\,\bbE\bigg[\frac{\big\vert f(\B_{t_s}^x) - f(\B_{t_s}^{\gamma_s})\big\vert}{\met(\B_{t_s}^x,\B_{t_s}^{\gamma_s})}\,\met(\B_{t_s}^x,\B_{t_s}^{\gamma_s})\,\One_{V_s}\bigg] \\
&\qquad\qquad \leq \Lip(f)\,\limsup_{s\downarrow 0} \frac{1}{t_s\,s^{3/2}}\, \bbE\Big[\met(\B_{t_s}^x,\B_{t_s}^{\gamma_s})^2\,\One_{A_s}\Big]\textcolor{white}{\bigg\vert}\\
&\qquad\qquad \leq \Lip(f)\,\limsup_{s\downarrow 0} \frac{s^2}{t_s\,s^{3/2}} \, \bbE\Big[\rme^{-\int_0^{t_s}\kkk\left(\B_r^x,\B_r^{\gamma_s}\right)\d r}\,\One_{A_s}\Big]\textcolor{white}{\bigg\vert}\\
&\qquad\qquad \leq \Lip(f)\,\limsup_{s\downarrow 0} \frac{s^{1/2}}{t_s}\, \rme^{-Kt_s} =0.
\end{align*}

\textbf{Step 3.3.} In a similar way, we can ignore the influence of $W_s$ by
\begin{align*}
&\limsup_{s\downarrow 0} \frac{1}{t_s\,s}\,\bbE\big[\big\vert f(\B_{t_s}^x) - f(\B_{t_s}^{\gamma_s})\big\vert\,\One_{W_s}\big]\\
&\qquad\qquad = \limsup_{s\downarrow 0} \frac{1}{t_s\,s}\,\bbE\bigg[\frac{\big\vert f(\B_{t_s}^x) - f(\B_{t_s}^{\gamma_s})\big\vert}{\met(\B_{t_s}^x,\B_{t_s}^{\gamma_s})}\,\met(\B_{t_s}^x,\B_{t_s}^{\gamma_s})\,\One_{W_s\cap\{X_{t_s}^x \neq X_{t_s}^{\gamma_s}\}}\bigg]\\
&\qquad\qquad \leq \Lip(f)\,\limsup_{s\downarrow 0}\frac{1}{t_s^2\,s^{3/2}}\,\bbE\Big[\!\int_0^{t_s}\!\met\big(\B_{t_s}^x,\B_{t_s}^{\gamma_s}\big)\,\met(\B_r^x,\B_r^{\gamma_s})\,\One_{A_s}\d r\Big]\\
&\qquad\qquad \leq \Lip(f)\,\limsup_{s\downarrow 0} \frac{s^2}{t_s^2\,s^{3/2}}\,\bbE\Big[\!\int_0^{t_s}\rme^{-\int_0^{t_s}\lll\left(X_a^x, X_a^{\gamma_s}\right)/2\d a}\,\rme^{-\int_0^{r}\lll\left(X_a^x, X_a^{\gamma_s}\right)/2\d a}\,\One_{A_s}\d r\Big]\\
&\qquad\qquad \leq \Lip(f)\,\limsup_{s\downarrow 0}\frac{s^{1/2}}{t_s}\,\rme^{-Kt_s} =0.
\end{align*}

\clearpage
\textbf{Step 3.4.} Finally we turn to the most delicate part, namely the study of the effect of $U_s$. To this aim, we first note that, defining the function $\ell\in \Lip_\bounded(\mms)$ by $\ell(x) := \lll(x,x)$, on the event $A_s \cap W_s^\rmc$ we have
\begin{align*}
&\int_0^{t_s}\lll(\B_r^x,\B_r^{\gamma_s})\d r - \int_0^{t_s}\ell(X_r^x)\d r\\
&\qquad\qquad = \int_0^{t_s}\lll(\B_r^x,\B_r^{\gamma_s})\d r - \int_0^{t_s}\lll(X_r^x, X_r^x)\d r\\
&\qquad\qquad \geq -\Lip(\lll)\int_0^{t_s}\met(X_r^x,X_r^{\gamma_s})\d r \geq -\Lip(\lll)\,t_s\,s^{1/2}.
\end{align*}
Together with \eqref{Eq:Pathwise bound maximum} and since $\met(\B_{t_s}^x,\B_{t_s}^{\gamma_s}) < s^{1/2}$ on $A_s\cap V_s^\rmc$, we thus obtain
\begin{align*}
&\limsup_{s\downarrow 0} \frac{1}{t_s}\,\Big[\,\frac{1}{s}\,\bbE\big[\big\vert f(\B_{t_s}^x) - f(\B_{t_s}^{\gamma_s})\big\vert\,\One_{U_s}\big] - \vert\nabla f(x)\vert\Big]\\
&\qquad\qquad\leq \limsup_{s\downarrow 0} \frac{1}{t_s}\,\bigg[\,\frac{1}{s}\,\bbE\bigg[\met(\B_{t_s}^x,\B_{t_s}^{\gamma_s})\,\frac{\big\vert f(\B_{t_s}^x) - f(\B_{t_s}^{\gamma_s})\big\vert}{\met(\B_{t_s}^x,\B_{t_s}^{\gamma_s})}\,\One_{U_s\cap \{X_{t_s}^x \neq X_{t_s}^{\gamma_s}\}}\bigg] - \vert\nabla f(x)\vert\bigg]\\
&\qquad\qquad \leq \limsup_{s\downarrow 0} \frac{1}{t_s}\,\bigg[\bbE\bigg[\rme^{-\int_0^{t_s}\ell(\B_r^x)/2\d r}\,\rme^{\Lip(\lll)\,t_s\,s^{1/2}/2}\\
&\qquad\qquad\qquad\qquad \times \sup_{z\in \Ball_{s^{1/2}}(\B_{t_s}^x)\setminus\{\B_{t_s}^x\}}\frac{\big\vert f(\B_{t_s}^x) - f(z)\big\vert}{\met(\B_{t_s}^x,z)}\,\One_{A_s}\bigg]  -\vert\nabla f(x)\vert\bigg].
\end{align*}
For small enough $s>0$, on the event $A_s$ we have $\smash{\Ball_{2s^{1/2}}(\B_{t_s}^x) \subset \Ball_{2\varepsilon}(x)}$. In this case, by applying  the mean value theorem twice,
\begin{align*}
&\sup_{z\in \Ball_{s^{1/2}}(\B_{t_s}^x)\setminus\{\B_{t_s}^x\}}\frac{\big\vert f(\B_{t_s}^x) - f(z)\big\vert}{\met(\B_{t_s}^x,z)}\\
&\qquad\qquad \leq \sup_{y\in \Ball_{2s^{1/2}}(\B_{t_s}^x)}\vert\nabla f(y)\vert\\
&\qquad\qquad \leq \vert\nabla f(\B_{t_s}^x)\vert + \sup_{y\in \Ball_{2s^{1/2}}(\B_{t_s}^x)}\big\vert \vert\nabla f(y) \vert - \vert\nabla f(\B_{t_s}^x)\vert\big\vert\\
&\qquad\qquad \leq \vert\nabla f(\B_{t_s}^x)\vert + 2\,s^{1/2}\sup_{v\in \overline{\Ball}_{4\varepsilon}(x)} \big\vert\nabla\vert\nabla f\vert(v)\big\vert.
\end{align*}
\clearpage
\noindent Let $\psi\in\Cont_\comp^\infty(\mms)$ be nonnegative with $\psi = \vert\nabla f\vert$ on $\Ball_{2\varepsilon}(x)$.  
Invoking the dominated convergence theorem, \eqref{Eq:Formula Pt BM} and the smoothness of the heat semigroup up to zero, the terms containing $s^{1/2}$ above become negligible as $s\downarrow 0$, and we are left with
\begin{align*}
&\limsup_{s\downarrow 0} \frac{1}{t_s}\,\Big[\,\frac{1}{s}\,\bbE\big[\big\vert f(\B_{t_s}^x) - f(\B_{t_s}^{\gamma_s})\big\vert\,\One_{U_s}\big] - \vert\nabla f(x)\vert\Big]\\
&\qquad\qquad \leq \limsup_{s\downarrow 0} \frac{1}{t_s}\,\Big[\bbE\Big[\rme^{-\int_0^{t_s}\ell(\B_r^x)/2\d r}\,\vert\nabla f(\B_{t_s}^x)\vert\,\One_{A_s}\Big] - \vert\nabla f(x)\vert\Big]\\
&\qquad\qquad = \limsup_{s\downarrow 0} \frac{1}{t_s}\,\Big[\bbE\Big[\rme^{-\int_0^{t_s}\ell(\B_r^x)/2\d r}\,\psi(\B_{t_s}^x)\,\One_{A_s}\Big] - \psi(x)\Big]\\
&\qquad\qquad \leq \limsup_{s\downarrow 0} \frac{1}{t_s}\,\big[\bbE\big[\psi(\B_{t_s}^x)\big] - \psi(x)\big] \\
&\qquad\qquad\qquad\qquad + \limsup_{s\downarrow 0} \frac{1}{t_s}\,\bbE\Big[\big[\rme^{-\int_0^{t_s}\ell(\B_r^x)/2\d r} - 1\big]\,\psi(\B_{t_s}^x)\Big]\\
&\qquad\qquad = \frac{1}{2}\,\Delta \psi(x) - \frac{1}{2}\,\ell(x)\,\psi(x) = \frac{1}{2}\,\Delta \vert\nabla f(x)\vert - \frac{1}{2}\,\ell(x)\,\vert\nabla f(x)\vert,
\end{align*}
where in second last identity, we used that the marginal law of $\B^x$  is independent of $s$. Since $\kkk$ was arbitrary, we conclude \eqref{Eq:1-Bochner} by Lemma \ref{Le:Approximation lemma}. 
\end{proof}

\begin{remark}\label{Re:Ru} Without \eqref{Eq:Exp integr assumption}, a careful inspection of the previous proof shows that if (iii) in Theorem \ref{Th:Equivalences} holds for any symmetric lower semicontinuous function $\kk\colon \mms\times\mms \to \R$ which is not necessarily the average of some function as in \eqref{Eq:kk definition}, then the $\Ell^1$-Bochner inequality holds for the function $k\colon \mms\to\R$ defined by $k(x) := \kk(x,x)$.
\end{remark}

\appendix

\section[Kato decomposable lower Ricci bounds and their Schrödinger semigroups]{Kato decomposable lower Ricci bounds and their\\Schrödinger semigroups}\label{Sec:Kato}

\subsection[The $\Ell^1$-gradient estimate]{The {\boldmath{$\Ell^1$}}-gradient estimate}\label{Sec:2}

In this section, we present a last equivalent characterization of the condition $\Ric\geq k$ on $\mms$ for the  class of Kato decomposable $k$ in terms of gradient estimates for $(\ChHeat_t)_{t\geq 0}$. A similar result can be found in \cite[Corollary 2.2]{wu2020}. See also \cite[Theorem 2.3.1]{wang2014} for more geometric growth conditions on $k^-$, and  \cite[Theorem 1.1]{braun2019} for the nonsmooth case under boundedness of $k^-$, the condition $\Ric\geq k$ on $\mms$ interpreted in a synthetic sense \cite{sturm2015}.

\begin{theorem}\label{Th:Sc} Assume that $k\colon \mms \to \R$ is a continuous and Kato decomposable function. Then any of the equivalent conditions in Theorem \ref{Th:Equivalences} is equivalent to the \emph{$\Ell^1$-gradient estimate w.r.t.~$k$}, i.e.~for every $f\in C_\comp^\infty(\mms)$,
\begin{equation}\label{TU}
\vert\nabla \ChHeat_t f(x)\vert \leq \bbE\Big[\rme^{-\int_0^t k(\B_r^x)/2\d r}\,\vert\nabla f\vert(\B_t^x)\,\One_{\{t<\zeta^x\}}\Big]\quad\text{for every }x\in\mms,\ t>0.
\end{equation}
\end{theorem}

\begin{proof} If $k$ obeys $\Ric \geq k$ on $\mms$, then the claimed $\Ell^1$-gradient estimate is just a restatement of Theorem \ref{Th:Form FK} for exact $1$-forms together with \eqref{Eq:Commutation}.

Conversely, assume the $\Ell^1$-gradient estimate. A similar argument as in the proof of (i) in Theorem \ref{Th:First theorem} in Section \ref{Sec:Stochastic completeness} -- directly employing \eqref{TU} instead of Theorem \ref{Th:Form FK} -- shows that $\mms$ is stochastically complete. Let $f\in\Cont_\comp^\infty(\mms)$ and $x\in\mms$ with $\vert\nabla f(x)\vert \neq 0$. Let $\varepsilon > 0$ such that $\vert \nabla f\vert$ is bounded away from zero -- in particular smooth -- on $\overline{\Ball}_\varepsilon(x)$. By Kato's inequality for the Bochner Laplacian \cite[Proposition 2.2]{hess1980},  we have $\vert\nabla f\vert\in W^{1,2}(\mms)$. Thus, given a nonnegative $\phi \in \Cont_\comp^\infty(\mms)$ with support in $\Ball_\varepsilon(x)$, by the chain rule, \eqref{Eq:Formula Pt BM} and a standard representation of the quadratic form $\mathscr{E}$ in terms of $(\ChHeat_t)_{t\geq 0}$, see e.g.~page 15 in \cite{davies},
\begin{align}\label{Eq:EQU}
&\frac{1}{2}\int_\mms \vert \nabla f(y)\vert^{-1}\,\big\langle\nabla f(y),\nabla \Delta f(y)\big\rangle\,\phi(y)\d\meas(y)\nonumber\\
&\qquad\qquad = \frac{\rmd}{\rmd t}\bigg\vert_{0} \int_\mms \big[\vert \nabla \ChHeat_t f(y)\vert^2\big]^{1/2}\,\phi(y)\d\meas(y)\nonumber\\
&\qquad\qquad = \lim_{t\downarrow 0} \frac{1}{t}\int_\mms \big[\vert\nabla\ChHeat_tf(y) \vert - \vert\nabla f(y)\vert\big]\,\phi(y)\d\meas(y)\nonumber\\
&\qquad\qquad\leq \limsup_{t\downarrow 0} \frac{1}{t}\int_\mms \Big[\bbE\Big[\rme^{-\int_0^t k(\B_r^y)/2\d r}\,\vert\nabla f(\B_t^y)\vert\Big] - \vert\nabla f(y)\vert\Big]\,\phi(y)\d\meas(y)\nonumber\\
&\qquad\qquad \leq \limsup_{t\downarrow 0}\frac{1}{t}\int_\mms \big[\bbE\big[\vert\nabla f(\B_t^y)\vert\big] - \vert\nabla f(y)\vert\big]\,\phi(y)\d\meas(y)\nonumber\\
&\qquad\qquad\qquad\qquad + \limsup_{t\downarrow 0} \frac{1}{t}\int_\mms \bbE\Big[\big[\rme^{-\int_0^t k(\B_r^y)/2\d r} - 1\big]\,\vert\nabla f(\B_t^y)\vert\Big]\,\phi(y)\d\meas(y)\nonumber\\
&\qquad\qquad = -\frac{1}{2}\int_\mms \big\langle\nabla \vert\nabla f\vert(y),\nabla\phi(y) \big\rangle\d\meas(y)\\
&\qquad\qquad\qquad\qquad + \limsup_{t\downarrow 0} \frac{1}{t}\int_\mms\bbE\Big[\big[\rme^{-\int_0^t k(\B_r^y)/2\d r}-1\big]\,\vert\nabla f\vert(\B_t^y)\Big]\,\phi(y)\d\meas(y).\nonumber
\end{align}
It remains to estimate the latter limit. Let $\tau^y_\varepsilon$ be the first exit time of $X^y$ from $\Ball_\varepsilon(y)$. Since $k$ is bounded on the bounded set $\smash{\bigcup_{y\in \overline{\Ball}_\varepsilon(x)}\overline{\Ball}_\varepsilon(y)}$, and by continuity of Brownian sample paths, the dominated convergence theorem gives
\begin{align*}
&\limsup_{t\downarrow 0} \frac{1}{t}\int_\mms\bbE\Big[\big[\rme^{-\int_0^t k(\B_r^y)/2\d r} - 1\big]\,\vert\nabla f\vert(\B_t^y)\,\One_{\{t< \tau_\varepsilon^y\}}\Big]\,\phi(y)\d\meas(y)\\
&\qquad\qquad = -\frac{1}{2}\int_\mms k(y)\,\vert\nabla f(y)\vert\,\phi(y)\d\meas(y).
\end{align*}
\clearpage
\noindent Using the Cauchy-Schwarz inequality, Lemma \ref{Le:Khas} and Lemma \ref{Le:Hsu lemma}, for arbitrary $T>0$ we obtain
\begin{align*}
&\limsup_{t\downarrow 0} \frac{1}{t}\int_\mms\bbE\Big[\big[\rme^{-\int_0^t k(\B_r^y)/2\d r} - 1\big]\,\vert\nabla f\vert(\B_t^y)\,\One_{\{t\geq \tau^y_\varepsilon\}}\Big]\,\phi(y)\d\meas(y)\\
&\qquad\qquad \leq \Vert\nabla f\Vert_{\Ell^\infty}\int_\mms \bbE\Big[\big\vert \rme^{\int_0^T k^-(X_r^y)/2\d r} - 1\big\vert^2\Big]^{1/2}\,\bbP\big[\tau_\varepsilon^y \leq t\big]^{1/2}\, \phi(y)\d\meas(y)\\
&\qquad\qquad \leq \sqrt{2}\,\Vert\nabla f\Vert_{L^\infty}\,\bigg[\sup_{y\in\mms} \bbE\Big[\rme^{\int_0^T k^-(\B_r^y)\d r}\Big]^{1/2} + 1\bigg]\\
&\qquad\qquad\qquad\qquad\times\limsup_{t\downarrow 0} \frac{1}{t}\int_\mms\bbP\big[\tau_\varepsilon^y \leq t\big]^{1/2}\,\phi(y)\d\meas(y)= 0,
\end{align*}
and the $L^1$-Bochner inequality \eqref{Eq:1-Bochner}  follows after integrating \eqref{Eq:EQU} by parts and using the arbitrariness of $\phi$.
\end{proof}

\begin{remark}\label{Re:Schröd} One can replace $C_\comp^\infty(\mms)$ by $W^{1,2}(\mms)$ in Theorem \ref{Th:Sc}. This follows from (\ref{Eq:Commutation}) and the fact that under Kato decomposability, the Feynman--Kac formula for the heat semigroup on $1$-forms, Theorem \ref{Th:Form FK}, holds for all square integrable $1$-forms. (This formula has been shown in \cite{guneysu2012} in a more general context, and of course it also follows from Theorem \ref{Th:Form FK} by approximating forms in $\Gamma_{L^2}(T^*\mms)$ by elements of $\Gamma_{C^\infty_c}(T^*\mms)$ using Lemma \ref{Le:Khasminskii}.) In view of the Cauchy--Schwarz inequality it seems unlikely that the Feynman--Kac formula on $1$-forms holds for all square integrable $1$-forms under the weaker assumption \eqref{Eq:Exp integr assumption}, although we are not aware of a counterexample (which would be interesting to have). We refer the reader also to the recent \cite{boldt}, where Feynman--Kac formulas for general perturbations of order no larger than $1$ -- rather than just zeroth order perturbations -- of Bochner Laplacians on vector bundles have been treated.
\end{remark}

\begin{remark} Somewhat in line with the previous remark, assume that $k$ satisfies  \eqref{Eq:Exp integr assumption} instead of Kato decomposability. Of course, if $\Ric\geq k$ on $\mms$, the $\Ell^1$-gradient estimate from Theorem \ref{Th:Sc} then still holds by virtue of Theorem \ref{Th:Form FK}. However, as it becomes apparent from the above proof, the converse implication seems to be more involved and to require at least some higher order exponential integrability of $k^-$.
\end{remark}

\subsection{Schrödinger semigroups} \label{Schr Sem}
For Kato decomposable $k$, the right-hand side of \eqref{TU} has a more analytic interpretation in terms of the \emph{Schrödinger semigroup} associated to $k$, which is briefly discussed now. Assume in this section that $k$ is a (not necessarily continuous) function which is Kato decomposable and in $L^2_\loc(\mms)$. Then $\Delta- k$ is essentially self-adjont in $L^2(\mms)$  \cite{guneysu2017'}, and the \emph{Schrödinger semigroup} $(\Schr{k}_t)_{t\geq 0}$ is defined to be $\smash{\Schr{k}_t := \rme^{t(\Delta-k)/2}}$ via spectral calculus. This is a strongly continuous semigroup of bounded linear operators in $\Ell^2(\mms)$. As $k$ is Kato decomposable, $(\Schr{k}_t)_{t\geq 0}$ has a pointwise well-defined version which, for every $f\in L^2(\mms)$, can be expressed \cite{ guneysu2017} via Brownian motion $\B^x$ on $\mms$ in terms of% the Feynman--Kac formula
\begin{equation}\label{Eq:Feynman-Kac}
\Schr{k}_{t}f(x) = \bbE\Big[\rme^{-\int_0^t k(\B_r^x)/2\d r}\, f(\B_t^x)\, \One_{\{t<\zeta^x\}}\Big]\forevery{x\in\mms,\ t\geq 0}.
\end{equation}

We are going to show that this semigroup extends to a strongly continuous semigroup of bounded operators in $\Ell^p(\mms)$ for all $p\in [1,\infty)$, see Theorem \ref{Cor:Lp properties}. To this end, we record \emph{Khasminskii's lemma} (which relies on the Markov property of the underlying diffusion on $\mms$).

\begin{lemma}\label{Le:Khasminskii} Let $\mathsf{v}\in\Kato(\mms)$. Then for every $\delta>1$ there exists a finite constant $C \geq 0$ depending only on $\vert \mathsf{v}\vert$ and $\delta$ such that
\begin{equation*}
\sup_{x\in\mms} \bbE\Big[\rme^{\int_0^t \vert \mathsf{v}(\B_r^x)\vert\d r}\, \One_{\{t<\zeta^x\}}\Big] \leq \delta\,\rme^{C t}\forevery{t\geq 0}.
\end{equation*}
\end{lemma}

Up to dealing with the small additional difficulty of stochastic incompleteness, the proof given in \cite{guneysu2015} follows from a careful examination of the classical Euclidean proof given in \cite{aizenman1982} (which however states a bound of the form $C_1\,\rme^{C_2 t}$). The above stronger bound has played a crucial role in the context of total variation considered in \cite{guneysu2015}. Analogous statements for general Hunt processes having the Feller property can also be found in \cite{demuth}.

\begin{theorem}\label{Cor:Lp properties} Let $k\colon \mms\to\R$ be a Kato decomposable function in $L^2_\loc(\mms)$. Then for every $\delta > 1$ there exists a finite constant $C\geq 0$  depending only on $k^-$ and $\delta$ such that, for every $p\in[1,\infty]$ and every $f\in L^2(\mms)\cap L^p(\mms)$, we have
\begin{equation}\label{Eq:Lq bounds on Schrödinger}
\big\Vert \Schr{k}_t f \big\Vert_{\Ell^p} \leq \delta\,\rme^{Ct}\,\Vert f\Vert_{\Ell^p}\forevery{t\geq 0}.
\end{equation}
In particular, for every $p\in[1,\infty]$, $\smash{(\Schr{k}_t)_{t\geq 0}}$ extends to a semigroup of bounded operators from $\Ell^p(\mms)$ into $\Ell^p(\mms)$ which indeed satisfies \eqref{Eq:Lq bounds on Schrödinger} for every $f\in\Ell^p(\mms)$ and, if $p<\infty$, is strongly continuous.
\end{theorem}

\begin{proof} The idea to prove \eqref{Eq:Lq bounds on Schrödinger} is to use  \eqref{Eq:Feynman-Kac} together with Lemma \ref{Le:Khasminskii} to show the desired inequality in the cases $p=\infty$ and $p=1$ (which needs an additional, but elementary exhaustion argument) and to apply Riesz--Thorin's theorem to extend it to all exponents $p\in[1,\infty]$. See \cite[Theorem IX.2, Corollary IX.4]{guneysu2017} for details.

%The existence of an extension of $(\Schr{k}_t)_{t\geq 0}$ to a semigroup of bounded operators from $\Ell^p(\mms)$ into $\Ell^p(\mms)$ for every $p\in [1,\infty]$ still satisfying \eqref{Eq:Lq bounds on Schrödinger} is standard by approximation, but we include the argument for the convenience of the reader, compare with the proof of Lemma \ref{Le:Linear operator}. For $p<\infty$ and $f\in\Ell^p(\mms)$, for any sequence $(f_n)_{n\in\N}$ in $\Ell^2(\mms)\cap\Ell^p(\mms)$ converging to $f$ in $\Ell^p(\mms)$, \eqref{Eq:Lq bounds on Schrödinger} implies that $(\Schr{k}_tf_n)_{n\in\N}$ is a Cauchy sequence in $\Ell^p(\mms)$. Thus, we define $\Schr{k}_tf$ as the $\Ell^p$-limit of the latter sequence as $n\to\infty$, and \eqref{Eq:Lq bounds on Schrödinger} again shows that this definition is independent of the choice of $(f_n)_{n\in\N}$. In the case $p=\infty$, given any reference point $o\in\mms$, the sequence $(f_n)_{n\in\N}$ defined by $f_n := f\,\One_{\Ball_n(o)}$ converges pointwise to $f$. By \eqref{Eq:Feynman-Kac} and Lemma \ref{Le:Khasminskii}, the dominated convergence theorem shows that the pointwise limit $\Schr{k}_tf$ of $(\Schr{k}_tf_n)_{n\in\N}$ as $n\to\infty$ is well-defined, and again this definition does not depend on the choice of $(f_n)_{n\in\N}$ once it is demanded that $\sup_{n\in\N}\Vert f_n\Vert_{\Ell^\infty}<\infty$. It is clear that both approximation procedures preserve  \eqref{Eq:Lq bounds on Schrödinger}.

The existence of an extension of $(\Schr{k}_t)_{t\geq 0}$ to a semigroup of bounded operators from $\Ell^p(\mms)$ into $\Ell^p(\mms)$ for every $p\in [1,\infty]$ still satisfying \eqref{Eq:Lq bounds on Schrödinger} is then standard by approximation. We include the argument for $p=\infty$ for the convenience of the reader. Given $f\in\Ell^\infty(\mms)$ and any reference point $o\in\mms$, the sequence $(f_n)_{n\in\N}$ defined by $f_n := f\,\One_{\Ball_n(o)}\in\Ell^2(\mms)\cap\Ell^\infty(\mms)$ converges pointwise to $f$. By \eqref{Eq:Feynman-Kac} and Lemma \ref{Le:Khasminskii}, the dominated convergence theorem shows that the pointwise limit $\Schr{k}_tf$ of $(\Schr{k}_tf_n)_{n\in\N}$ as $n\to\infty$ is well-defined. This definition does not depend on the choice of $(f_n)_{n\in\N}$ as long as $\sup_{n\in\N}\Vert f_n\Vert_{\Ell^\infty}<\infty$. It is also clear that this procedure preserves  \eqref{Eq:Lq bounds on Schrödinger}.

To show strong continuity of $(\Schr{k}_t)_{t\geq 0}$ in $\Ell^p(\mms)$ for $p<\infty$, by approximation and \eqref{Eq:Lq bounds on Schrödinger}, it suffices to show continuity of $\smash{t\mapsto\Schr{k}_tf}$ on $[0,\infty)$ in $\Ell^p(\mms)$ for $f\in \Ell^2(\mms)\cap\Ell^p(\mms)\cap\Ell^\infty(\mms)$. By the semigroup property, we may and will restrict ourselves to the proof of continuity at $t=0$. Given any $x\in\mms$, note that a.s., we have $\smash{\int_0^t k(\B_r^x)\d r \to 0}$ as $t\downarrow 0$ since $k\in\Ell^2_\loc(\mms)$, and that
\begin{equation}\label{Eq:Boundbound}
\Big\vert \rme^{-\int_0^t k(\B_r^x)/2\d r} - 1\Big\vert \leq \rme^{\int_0^T k^-(\B_r^x)/2\d r} + 1
\end{equation}
for every $t\in [0,T]$ is satisfied a.s.~for fixed $T>0$. Since
\begin{equation*}
\int_\mms \big\vert \Schr{k}_t f-\ChHeat_t f \big\vert^p\d\meas \leq \int_\mms \bbE\Big[\Big\vert \rme^{-\int_0^t k(\B_r^x)/2\d r} - 1\Big\vert\,\vert f(\B_t^x)\vert\,\One_{\{t<\zeta^x\}}\Big]^p\d\meas(x),
\end{equation*}
applying the dominated convergence theorem twice using \eqref{Eq:Boundbound} as well as Lemma \ref{Le:Khasminskii}, we obtain $\Schr{k}_t f - \ChHeat_tf \to 0$ in $\Ell^p(\mms)$ as $t\downarrow 0$. The result follows immediately by strong continuity of the heat flow $(\ChHeat_t)_{t\geq 0}$ in $\Ell^p(\mms)$.
\end{proof}

We close this section by noting that, using quadratic form techniques \cite{demuth, guneysu2017, stollmann1996} in order to define the Schrödinger operator, it is possible to treat $L^1_\loc$-potentials $k$ rather than $L^2_\loc$ along the same lines.

\subsection{Proof of Theorem \ref{ayx}}\label{Check}

Now, we present one possible step-by-step analysis in order to check the existence of (continuous) Kato decomposable lower Ricci bounds for $\mms$, along with proving Theorem \ref{ayx}. Let us abbreviate $d:=\dim(\mms)$.

\begin{proof}[Proof of Theorem \ref{ayx}] Let $\Xi:M\to (0,\infty)$ be a Borel function such that
	\begin{equation}\label{Ugh}
	\sup_{y\in M} \sfp_t(x,y)\leq \Xi(x)\,\big[t^{-d/2}+1]\quad\text{for every }x\in\mms,\ t\in (0,1].
	\end{equation}
	(Using a parabolic $L^1$-mean value inequality, it has been shown in \cite[Theorem 2.9]{guneysu2017'}, see also \cite[Remark IV.17]{guneysu2017}, that every Riemannian manifold admits a canonical choice of a function $\Xi$ as above.)  \cite[Proposition VI.10]{guneysu2017} states that for every $p\in [1,\infty)$, if $d=1$, and every $p\in (d/2,\infty)$, if $d\geq 2$, we have $L^p_{\Xi}(\mms)+L^{\infty}(\mms)\subset \Kato(\mms)$. Thus, any locally $\meas$-integrable function $k\colon\mms\to\mathbb{R}$ such that 
	\begin{equation*}
	k^-\in L^p_{\Xi}(\mms)+L^{\infty}(\mms)
	\end{equation*}
	for some $\Xi$ and $p$ as above is Kato decomposable. 
	
	Now let $\langle\cdot,\cdot\rangle$ be quasi-isometric to a complete metric on $\mms$ whose Ricci curvature is bounded from below by constant. Then, as the Li--Yau heat kernel estimate, the Cheeger--Gromov volume estimate and the local volume doubling property are qualitatively stable under quasi-isometry, it follows from the considerations in \cite[Example IV.18]{guneysu2017} that there exists a constant $C>0$ such that
	\begin{equation*}
	\mathsf{p}_t(x,y)\leq C\, \meas\big[\Ball_1(x)\big]^{-1}\, \big[t^{-d/2}+1\big]\quad\text{for every }x,y\in\mms,\ t\in (0,1].
	\end{equation*}
	Thus every $k:\mms\to\mathbb{R}$ such that, choosing $\smash{\Xi:= \meas\big[\Ball_1(\cdot)\big]^{-1}}$, one has
	\begin{equation*}
	k^-\in L^p_{\Xi}(\mms)+L^{\infty}(\mms)
	\end{equation*}
	for some $p$ as in the previous step is Kato decomposable.
\end{proof}

\begin{remark} The previous proof shows that the assertion of Theorem \ref{ayx} remains valid if the inverse volume function is replaced by any function obeying \eqref{Ugh}.
\end{remark}

\begin{example}\label{Exa} Assume that $\mms$ is a model manifold in the sense of \cite{grigoryan2009}, meaning that $\mms=\mathbb{R}^d$ as a manifold with $d\geq 2$, and that the Riemannian metric $\langle\cdot,\cdot\rangle$ is given in polar coordinates as $\rmd r^2 +\psi(r)\d \theta^2$, where $\psi$ is a smooth positive function on $(0,\infty)$. The volume of balls on such manifolds does not depend on the center, and the Ricci curvature behaves in the radial direction like $\psi''/\psi - (d-1)(\psi')^2/\psi^2$, see e.g.~page 266 in \cite{besse}. Assume now
\begin{align*}\label{aspq}
\big(\psi''/\psi-(d-1)(\psi')^2/\psi^2\big)^-\in L^p_{\psi^{d-1}}((0,\infty))+L^{\infty}((0,\infty))\quad\text{for some }p>d/2,
\end{align*}
where $\smash{\Ell^p_{\psi^{d-1}}((0,\infty))}$ is the $\Ell^p$-space of functions w.r.t.~$\smash{\One_{(0,\infty)}\,\psi^{d-1}\,\mathscr{L}^1}$. As the volume measure behaves in the radial direction as $\psi^{d-1}(r)\d r$, it follows that the Ricci curvature is bounded from below by a function with negative part in $L^p(\mms)+L^{\infty}(\mms)$. 

To make sure that the latter function space is included in $\Kato(\mms)$ it suffices from the above considerations to assume that there exists a smooth positive function $\psi_0$ on $(0,\infty)$ such that 
\begin{enumerate}[leftmargin=1.25cm,label=\alph*.]
\item $\psi_0(0)=0$, $\psi_0'(0)=1$ and $\psi_0''(0)=0$,
\item $\psi_0''/\psi_0-(d-1)(\psi_0')^2/\psi_0^2$ is uniformly bounded from below by a constant, and
\item $\psi_0/C\leq \psi\leq C\psi_0$ for some constant $C>1$.
\end{enumerate}  
Indeed, a.~guarantees that there exists a complete metric $g_0$ on $\mms$ which -- in polar coordinates -- is written as $g_0=\rmd r^2 +\psi_0(r)\d \theta^2$. Assumption b.~guarantees that the Ricci curvature associated to $g_0$  is bounded from below by a constant, and c.~implies that $g$ is quasi-isometric to $g_0$. For instance, one can take the Euclidean metric corresponding to $\psi_0(r):=r$ or the hyperbolic metric corresponding to $\psi_0(r)=\sinh(r)$ as reference metrics.
\end{example}

\end{document}